\documentclass{amsart}
\usepackage{amsmath,amscd}
\usepackage{amsfonts}
\usepackage[all]{xy}
\usepackage{enumerate}
\numberwithin{equation}{subsection}
\theoremstyle{plain}
\newtheorem{thm}[subsection]{Theorem}
\newtheorem{prop}[subsection]{Proposition}
\newtheorem{lemma}[subsection]{Lemma}
\newtheorem{cor}[subsection]{Corollary}
\theoremstyle{definition}
\newtheorem{defn}[subsection]{Definition}
\newtheorem{notn}[subsection]{Notation}
\newtheorem{cont}[subsection]{Contents}
\newtheorem{ackn}[subsection]{Acknowledgement}
\theoremstyle{remark}
\newtheorem{rem}[subsection]{Remark}

\begin{document}
\title{The Manin constant of elliptic curves over function fields}
\author{Ambrus P\'al}
\date{January 9, 2010.}
\address{Department of Mathematics, 180 Queen's Gate, Imperial College, London, SW7 2AZ, United Kingdom}
\email {a.pal@imperial.ac.uk}
\begin{abstract} We study the $p$-adic valuation of the values of normalised Hecke eigenforms attached to non-isotrivial elliptic curves defined over function fields of transcendence degree one over finite fields of characteristic $p$. We derive upper bounds on the smallest attained valuation in terms of the minimal discriminant under a certain assumption on the function field
and provide examples to show that our estimates are optimal. As an application of our results we also prove the analogue of the degree conjecture unconditionally for strong Weil curves with square-free conductor defined over function fields satisfying the assumption mentioned above.
\end{abstract}
\maketitle

\section{Introduction}

\begin{notn} Let $F$ denote the function field of $\mathcal C$, where the latter is a geometrically connected smooth projective curve defined over the finite field $\mathbb F_q$ of characteristic $p$. Let $\mathbb A$ denote the ring of ad\`eles of $F$ and let $GL_2$ denote the group scheme of invertible two by two matrices. Let $E$ be a non-isotrivial elliptic curve defined over $F$. Then we may associate a cuspidal automorphic representation of $GL_2(\mathbb A)$ to $E$ as follows. Let $E_{\overline F}$ denote the base change of $E$ to the separable closure $ \overline F$ of $F$. For every prime $l$ different from $p$ one may attach to the \'etale cohomology group $H^1(E_{ \overline F},\mathbb Q_l(1))$, considered as a representation of the absolute Galois group of $F$, an irreducible cuspidal automorphic representation $\rho_E$ with trivial central character via the Langlands correspondence. As the notation indicates this representation is independent of the choice of $l$. 

Let $V_E$ denote the irreducible constituent of the space of cuspidal automorphic forms on $GL_2(\mathbb A)$ which realises the representation $\rho_E$. Then there is a distinguished element $\psi_E$ of $V_E$ which we will call the normalised Hecke eigenform attached to $E$. It is characterised by the fact that it is invariant under the action of the Hecke congruence group of level $\mathfrak n$, where $\mathfrak n$ denotes the conductor of the elliptic curve $E$, and its leading Fourier coefficient is $1$. (For an explanation of
these concepts as well as an explicit description of the Hecke eigenform see the following chapter.) By a classical theorem of Harder the automorphic form $\psi_E$ takes only finitely many values. On the other hand it is easy to see that it takes only rational values. Hence there is a unique positive rational number $c(E)$ such that the subgroup of $\mathbb Q$ generated by the values of $\psi_E$ is equal to $c(E)\mathbb Z$. Then the following holds: 
\end{notn}
\begin{prop} There is a natural number $m(E)$ such that
$c(E)=p^{-m(E)}$.
\end{prop}
It is natural to guess that $m(E)$, which we will call the Manin constant of $E$, is always zero. Although this hypothesis is frequently made (sometimes implicitly) in the literature (see for example [20], [21] and [23]) it is actually false. One of the aims of this paper is to exhibit many cases when $m(E)$ is not zero. Because the Manin constant could be non-zero many formulas in the literature have to be corrected to include this non-trivial factor. Hence the latter is a very interesting isogeny invariant of the elliptic curve, therefore it is desirable to compute it, or at least to give upper bounds, in terms of more well-known invariants. This is the other major aim of this paper. (This problem has been already studied by Tan in [28], see Remark 7.9 below). We will also discuss the implication of our results in relation with one of the formulas mentioned above. 

Let us formulate now the main results of this paper. For every $E$ as above let $\Delta_E$ denote the discriminant of a relatively minimal elliptic surface $\mathcal E\rightarrow\mathcal C$ whose generic fibre is $E$. Then $\Delta_E$ is an effective divisor on the curve $\mathcal C$. Moreover let $g$ denote the genus of $\mathcal C$ and let $d$ be the positive integer such that $q=p^d$. We will show the following:
\begin{thm} Assume that $p$ does not divide the order of $\text{\rm Pic}_0(\mathcal C)(\mathbb F_q)$. Then we have:
$$m(E)\leq d(\frac{1}{12}\deg(\Delta_E)+g-1),$$
and the two sides of the inequality above are equal when the elliptic surface $\mathcal E$ is ordinary in dimension $2$.
\end{thm}
The condition on $\text{\rm Pic}_0(\mathcal C)(\mathbb F_q)$ in the theorem above is satisfied for example when $\mathcal C$ is a rational curve or a supersingular elliptic curve. Moreover the moduli space of smooth projective connected curves of genus $g$ with $p$-rank zero is a variety of dimension $2g-3$ over $\overline{\mathbb F}_p$ when $g\geq2$ and $p$ is odd (see Theorem 2.3 and Proposition 2.7 of [4] on pages 120 and 122, respectively). Hence there are plenty of curves satisfying this condition. It is natural to expect that most elliptic surfaces are ordinary in dimension $2$ (for a precise formulation of this conjecture see Remark 6.12). In particular our estimate in Theorem 1.3 should be the best possible (at least if we want to make one in terms of the discriminant). We can verify the ordinariness condition in many cases. The following result is just a sample of what can be proven with our methods. 
\begin{thm} Let $p$ be a prime number and let $n$ be a positive integer such that $n|p-1$ and $6|n$. Let $E$ be the elliptic curve defined over the rational function field $F=\mathbb F_p(T)$ by the Weierstrass equation
$$y^2+xy=x^3-T^n.$$
Then $E$ is not isotrivial and 
$$m(E)=\frac{n}{6}-1=\frac{1}{12}\deg(\Delta_E)-1.$$
\end{thm}
The basic strategy of the proof of Theorem 1.3 is to relate the Manin constant to the $p$-adic valuation of coefficients of $L$-functions of $E$. The key tools in estimating the latter are a mild equivariant extension of Katz's conjecture relating the Newton and Hodge polygons and a theorem of Chinburg computing the refined equivariant Euler characteristic of the de Rham complex of varieties equipped with tame group actions in terms of $\epsilon$-constants. The proof of Theorem 1.4 is closely related. In fact the reason why it is particularly convenient to work with those elliptic curves which appear in the theorem is that Ulmer computed their Hasse-Weil $L$-functions rather explicitly in [29]. 

In the rest of the introduction we describe the application of Theorem 1.3 in this paper which was the main motivation for our investigations. Fix a closed point $\infty$ of $\mathcal C$ and assume that $E$ has split multiplicative reduction at $\infty$. Then $\mathfrak n=\mathfrak m\infty$ for an effective divisor $\mathfrak m$ on $\mathcal C$ where we write the addition of divisors multiplicatively in this paper. Let $A$ denote the ring of rational functions on $\mathcal C$ regular away from $\infty$ and let $X_0(\mathfrak m)$ denote the unique smooth projective curve over $F$ which contains the affine Drinfeld modular curve $Y_0(\mathfrak m)$ parameterising Drinfeld $A$-modules of rank two of generic characteristic with Hecke level $\mathfrak m$-structure as a dense open subscheme. Then there is a non-trivial map $\pi:X_0(\mathfrak m)\rightarrow E$ of curves defined over $F$. We say that $E$ is a strong Weil curve if the modular parameterisation $\pi$ above can be chosen so that the kernel of the map induced by $\pi$ via Albanese functoriality is smooth and connected in the Jacobian of $X_0(\mathfrak m)$. In this case we say that $\pi$ is optimal. Up to isomorphism there is exactly one strong Weil curve in the isogeny class of $E$. With the help of Theorem 1.3 and the Pesenti-Szpiro inequality we will show the following
\begin{thm} Assume that $p$ does not divide the order of $\text{\rm Pic}_0(\mathcal C)(\mathbb F_q)$. Also suppose that $\pi$ is optimal and $\mathfrak m$ is square-free. Then we have:
$$\deg(\pi)<q^{18g+4\deg(\infty)+1}\cdot q^{2\deg(\mathfrak m)}\cdot
\deg(\mathfrak m)^3.$$ 
\end{thm}
The result above is an analogue of Frey's celebrated degree conjecture, formulated originally for strong Weil curves over $\mathbb Q$. Our result completes the work of Papikian (see [21]). In his paper he made a conjecture which implies that the Manin constant is zero (at least when $F$ is the rational function field and $\infty$ is the point at infinity) and he derived an inequality significantly stronger then the one in Theorem 1.5 under this assumption. As we saw his hypothesis does not hold in general. In fact it is quite reasonable to expect in light of the above that $m(E)$ is $d(\frac{1}{12}\deg(\Delta_E)+g-1)$ when $E$ is minimal in its isogeny class (for the definition of the latter see chapter 7). This is reflected by the fact that the contribution of our estimate of the term $c(E)^{-2}$ to our bound on the degree of the modular parameterisation is significant, it is of the size $O(q^{\deg(\mathfrak m)})$.
\begin{cont} In the following chapter we will give an explicit  description of the Hecke eigenform and prove Proposition 1.2. In the third chapter we will work out carefully in detail the analogue of the theory of modular symbols for function fields, something which is missing from the current literature. Using these results we derive lower and upper bounds on the Manin constant in terms of the $p$-adic valuation of coefficients of $L$-functions of $E$ twisted with tamely ramified abelian characters in the fourth chapter. The aim of the fifth chapter is to relate the Galois module structure of the second coherent cohomology of the structure sheaf of elliptic surfaces equipped with a group action respecting the elliptic fibration to $\epsilon$-constants of Galois representations of the function field of the base in a special case. We prove a mild equivariant extension of Katz's conjecture relating the Newton and Hodge polygons and with its aid we derive Theorem 1.3 from our previous results in the sixth chapter. In the seventh chapter we show that the isogeny class of $E$ contains an elliptic curve whose $j$-invariant is not a $p$-th power then we use this result and the Pesenti-Szpiro inequality to deduce a bound on $m(E)$ in terms of the degree of the conductor of $E$ from Theorem 1.3. In the eighth chapter we first review the work of Ulmer in [29] then prove Theorem 1.4 by using the latter and a classical result of Stickelberger on $p$-adic valuations of Gauss sums. We show that the usual characterisation of strong Weil curves and optimal modular parameterisations holds in the function field setting as well in the ninth chapter. In the last chapter we first show that a certain homomorphism defined by Gekeler and Reversat in [5] has finite cokernel of exponent dividing $q^{\deg(\infty)}-1$. Then we combine this result with the bound in chapter seven and the work of Papikian to show Theorem 1.5.
\end{cont}
\begin{ackn} The author was partially supported by the EPSRC grant P19164. 
\end{ackn}

\section{The normalised Hecke eigenform}

\begin{notn} Let $\mathcal O$ denote the maximal compact subring of the ring $\mathbb A$ of ad\`eles of $F$. Let $|\mathcal C|$ denote the set of closed points of $\mathcal C$. For every ad\`ele $a\in\mathbb A$ and $x\in|\mathcal C|$ let $a_x$ denote the $x$-th component of $a$. Let $\mu$, $\mu^*$ be Haar measures on the locally compact abelian topological groups $\mathbb A$ and $\mathbb A^*$, respectively. Also assume that $\mu(\mathcal O)$ and $\mu^*(\mathcal O^*)$ are both equal to $1$. Since these measures are left-invariant with respect to the discrete subgroups $F^*$ and $F$ by definition, they induce a measure on $F^*\backslash\mathbb A^*$ and $F\backslash\mathbb A$, respectively, which will be denoted by the same letter by abuse of notation. For every divisor $\mathfrak m$ on $\mathcal C$ let $\mathfrak m\mathcal O$, $\mathbb K_0 (\mathfrak m)$ denote the sub $\mathcal O$-module of $\mathbb A$ generated by those id\`eles whose divisor is $\mathfrak m$, and the Hecke congruence subgroup of $GL_2(\mathbb A)$ of level $\mathfrak m$:
$$\mathbb K_0(\mathfrak m)=\{\left(\begin{matrix} a&b\\
c&d\end{matrix}\right)\in GL_2(\mathcal O)|
c\in\mathfrak m\mathcal O\}\text{,}$$
respectively. Let $\tau:F\backslash\mathbb A\rightarrow\mathbb C^*$ be a non-trivial continuous additive character. The composition of the quotient map $\mathbb A\rightarrow F\backslash\mathbb A$ and $\tau$ will be denoted by the same symbol by the usual abuse of notation. Let $\mathcal D$ denote the $\mathcal O$-module $\mathcal D=\{x\in\mathbb A|\tau(x\mathcal O)=1\}$, and let $\mathfrak d$ be a divisor on $\mathcal C$ such that $\mathcal D=\mathfrak d\mathcal O$. Let $\mathfrak n$ be the conductor of $E$. The latter is an effective divisor on $\mathcal C$. We may assume that the divisor $\mathfrak d$ is relatively prime to $\mathfrak n$ by changing $\tau$, if necessary. Let $B$ denote the group scheme of invertible upper triangular two by two matrices. Let $P$ denote the group scheme of invertible upper triangular two by two matrices with $1$ on the lower right corner. Finally let $Z$ denote the centre of the group scheme $GL_2$.
\end{notn}
\begin{defn} For every id\`ele $u\in\mathbb A^*$ let $(u)$
denote the corresponding divisor on $\mathcal C$. Often we will denote $(u)$ simply by $u$ by slight abuse of notation, when this does not cause confusion. We will call two divisors $\mathfrak m$ and $\mathfrak n$ on $\mathcal C$ relatively prime if their support is disjoint. Let Div$(\mathcal C)$ denote the group of divisors on $\mathcal C$. We will call a function $f:\text{Div}(\mathcal C)\rightarrow\mathbb C$ multiplicative if it vanishes on non-effective divisors, $f(1)=1$ and for every pair of relatively prime divisors $\mathfrak n$ and $\mathfrak m$ we have $f(\mathfrak n\mathfrak m)=f(\mathfrak n)f(\mathfrak m)$. Let $E$ be a non-isotrivial elliptic curve defined over $\mathcal C$ of conductor $\mathfrak n$. For every divisor $\mathfrak r$ on $\mathcal C$ let $\deg(\mathfrak r)$ denote the degree of $\mathfrak r$. For every $x\in|\mathcal C|$ let $L_x(E,t)$ denote the local factor of the Hasse-Weil $L$-function of $E$ at $x$. The latter can be written in the form:
$$L_x(E,t)=\sum_{n=0}^{\infty}a(x^n)(tq)^{n\deg(x)}
\in\mathbb Z[[t]],$$
for some $a(x^n)\in\mathbb Z[\frac{1}{p}]$. Let $a$ denote the unique multiplicative function into the multiplicative semigroup of $\mathbb Q$ such that $a(x^n)$ is the same as above for each natural number $n$ and each $x\in|\mathcal C|$. A continuous function $\psi_E:GL_2(\mathbb A) \rightarrow \mathbb Q$ is called a normalised Hecke eigenform attached to $E$ is if it satisfies the following properties:
\begin{enumerate}
\item[$(a)$] it is automorphic: $\psi_E(\gamma h)=\psi_E(h)$
for all $\gamma\in GL_2(F)$,
\item[$(b)$] it has trivial central character: $\psi_E(hz)=\psi_E(h)$
for all $z\in Z(\mathbb A)$,
\item[$(c)$] it is right $\mathbb K_0(\mathfrak n)$-invariant:
$\psi_E(hk)=\psi_E(h)$ for all $k\in\mathbb K_0(\mathfrak n)$,
\item[$(d)$] it is cuspidal:
$$\int_{F\backslash\mathbb A}\psi_E(\left(\begin{matrix} 1&x\\0&1\end{matrix}\right)h)d\mu(x)=0\text{\ for all $h\in GL_2(\mathbb A)$,}$$
\item[$(e)$] its Fourier coefficients are $a$:
$$a(\mathfrak m\mathfrak d)=\mu(F\backslash\mathbb A)^{-1}
\int_{F\backslash\mathbb A}\psi_E(\left(\begin{matrix} \overline{\mathfrak m}&x\\
0&1\end{matrix}\right))\tau(-x)d\mu(x)
\text{\ for all $\mathfrak m\in$Div$(\mathcal C)$,}$$
where $\overline{\mathfrak m}\in\mathbb A^*$ and $(\overline{\mathfrak
m})=\mathfrak m$.
\end{enumerate}
Note that the last two conditions make sense because of $(a)$ we may (and will) consider $\psi_E$ as a function on $GL_2(F)\backslash GL_2(\mathbb A)$ as well.
\end{defn}
\begin{prop} There is a unique normalised Hecke eigenform
attached to $E$.
\end{prop}
The claim above is certainly very well known and the only fact which needs an additional argument is that $\psi_E$ takes only rational values. As we will shortly prove a stronger claim we will omit the proof. By a classical theorem of Harder (see [7]) the normalised Hecke eigenform $\psi_E$ is supported on a finite set as a function on the double coset $GL_2(F)\backslash GL_2(\mathbb A)/\mathbb K_0(\mathfrak n)Z(\mathbb A)$. Let $L(E)\subseteq\mathbb Q$ denote the $\mathbb Z$-module generated by the values of $\psi_E$. By the above there is a unique positive rational number $c(E)\in\mathbb Q$ such that $L(E)=c(E)\mathbb Z$.
\begin{prop} There is a non-negative natural number $m(E)$
such that $c(E)=p^{-m(E)}$.
\end{prop}
As we already mentioned in the introduction we will call $m(E)$ the Manin constant of the elliptic curve $E$.
\begin{proof} First we are going to show that $L(E)\subseteq
\mathbb Z[\frac{1}{p}]$. By the approximation theorem we have $GL_2(\mathbb A)=GL_2(F)P(\mathbb A)Z(\mathbb A)\mathbb K_0(\mathfrak n)$. Therefore it will be sufficient to prove that $\psi_E(h)\in\mathbb Z[\frac{1}{p}]$ for every element $h$ of $P(\mathbb A)$. By the definition of Fourier coefficients:
$$\psi_E(\left(\begin{matrix} y&x\\0&1\end{matrix}\right))=
\sum_{\eta\in F^*}a(\eta y\mathfrak d^{-1})\tau(\eta x)=
\sum_{\eta\in S}a(\eta y\mathfrak d^{-1})(\sum_{\epsilon\in\mathbb F^*_q}
\tau(\eta\epsilon x))$$
for all $y\in\mathbb A^*$, $x\in\mathbb A$, where $S$ is a set of representatives of the quotient $\mathbb F^*_q\backslash F^*$, since $\psi_E$ is cuspidal. The character sums on the right hand side are all equal to $-1$ or $q-1$. Moreover the sum above is finite. As $a(\mathfrak m)\in\mathbb Z[\frac{1}{p}]$ for every effective divisor $\mathfrak m$, the claim is now clear. Now we only need to show that $-1\in L(E)$ in order to prove the proposition. Let
$ \overline{\mathfrak d}\in\mathbb A^*$ be such that $(\overline{\mathfrak d})=\mathfrak
d$. Then the Fourier expansion says:
$$\psi_E(\left(\begin{matrix}\overline{\mathfrak d}&x\\0&1\end{matrix}\right))=
\sum_{\eta\in F^*}a(\eta)\tau(\eta x)=\sum_{\epsilon\in\mathbb F_q^*}
a(\epsilon)\tau(\epsilon x)=\sum_{\epsilon\in\mathbb F_q^*}
\tau(\epsilon x),$$
because $F^*\cap\mathcal O=\mathbb F_q^*$ and $a(1)=1$ by definition. The character
sum on the right hand side is equal to $-1$ if
$x\notin\mathcal D$. 
\end{proof}

\section{Epsilon constants and toric integrals}

\begin{notn} For the rest of the paper fix a prime number $l$ different from $p$. By the axiom of choice we may pick an isomorphism $\iota:\overline{\mathbb Q}_l\rightarrow\mathbb C$. We will identify $\overline{\mathbb Q}_l$ with $\mathbb C$ via $\iota$ in all that follows. Let $E_l$ be a finite extension of $\mathbb Q_l$ and let $\rho:\text{Gal}(\overline F|F)\rightarrow GL_{E_l}(W)$ be an $l$-adic representation on a finite dimensional vector space $W$ over $E_l$. Moreover let $L(\rho,t)\in E_l[[t]]$ denote the Grothendieck $L$-function associated to $\rho$ as defined in 9.1 of [2] on page 574. By a classical theorem of Grothendieck (\S 10 of [2] on pages 581-592) the series $L(\rho,t)$ is a rational function in the variable $t$ and also satisfies the functional equation:
$$L(\rho,t)=\epsilon(\rho)t^{\alpha(\rho)}L(\rho^{\vee},q^{-1}t^{-1})$$
where $\alpha(\rho)\in\mathbb Z$, $\epsilon(\rho)\in E_l^*$ and $\rho^{\vee}$ is the degree of the conductor of $\rho$ (in the sense of [13], page 179),
the $\epsilon$-constant of $\rho$, and the dual $l$-adic representation on
$\text{Hom}(W,E_l)$, respectively.
\end{notn}
\begin{notn} Let $K$ be a local field, let $dx$ be a Haar measure on $K$, and let $\psi$ be a non-trivial additive character on $K$. For every continuous homomorphism $\alpha:K^*\rightarrow\mathbb C^*$ let $\epsilon(K,\alpha,\psi,dx)$ denote the local $\epsilon$-factor attached to the triple $(\alpha,\psi,dx)$ as defined in 3.3 of [2] on page 526. Let $W(\overline K|K)<\text{Gal}(\overline K|K)$ denote the Weil group of $K$ (as defined in 2.2.4 of [2] on page 522). Local class field theory furnishes an isomorphism $j:K^*\rightarrow W(\overline K|K)^{ab}$. We normalise this isomorphism so that for every uniformizer $\pi\in K^*$ the image of $j(\pi)$ with respect to the map $W(\overline K|K)\rightarrow\mathbb Z$ introduced in 2.2.4 of [2] is the geometric Frobenius (similarly to 2.3 of [2] on page 523). For every homomorphism $\alpha:\text{Gal}(\overline K|K)\rightarrow
\mathbb C^*$ let the same symbol $\alpha$ denote the composition of $j$, considered here as an imbedding $j:K^*\rightarrow\text{Gal}(\overline K|K)^{ab}$, with the map $\text{Gal}(\overline K|K)^{ab}\rightarrow\mathbb C^*$ induced by the character $\alpha$ by slight abuse of notation. 
\end{notn}
\begin{notn} For every $x\in|\mathcal C|$ let $F_x$, $\mathcal O_x$ denote the completion of $F$ at $x$ and the valuation ring of $F_x$, respectively. For every $x\in|\mathcal C|$ let $\mu_x$, $\mu^*_x$ be the unique Haar measures on the locally compact abelian topological groups $F_x$ and $F_x^*$, respectively, so that $\mu(\mathcal O_x)$ and $\mu^*(\mathcal O_x^*)$ are both equal to $1$. Moreover for every $x\in|\mathcal C|$ let $\tau_x:F_x\rightarrow\mathbb C^*$ be the unique continuous additive character such that $\tau(a)=\prod_{x\in|\mathcal C|}\tau_x(a_x)$ for every $a\in\mathbb A$. For every homomorphism $\alpha:\text{Gal}(\overline F|F)\rightarrow\mathbb C^*$ and for every $x\in|\mathcal C|$ let $\alpha_x:\text{Gal}(\overline F_x|F)\rightarrow\mathbb C^*$ denote the restriction of $\alpha$ onto the decomposition group at $x$. 
\end{notn}
\begin{thm} For every continuous homomorphism $\alpha:\text{\rm Gal}(\overline K|K)\rightarrow\mathbb C^*$ with finite image we have:
$$\epsilon(\alpha)=q^{1-g}
\prod_{x\in|\mathcal C|}\epsilon(F_x,\alpha_x,\tau_x,\mu_x).$$
\end{thm}
\begin{proof} This is just a special case of Th\'eor\`eme 3.2.1.1 of [13] on page 187.
\end{proof}
\begin{defn} For every divisor $\mathfrak m$ on $\mathcal C$ let $\underline{\mathfrak m}$ denote the support of $\mathfrak m$. Let $c\in\mathbb A^*$ be an id\`ele so that $\mathfrak c=(c)$ is an effective divisor on $\mathcal C$. For every such $c$ let $\widehat c\in\mathbb A$ be the unique ad\`ele such that $\widehat c_x=c_x^{-1}$  for every $x\in\underline{\mathfrak c}$, and $\widehat c_x=0$, otherwise. Let $\alpha:F\backslash\mathbb A^*\rightarrow\mathbb C^*$ be a continuous character with finite image whose conductor $\mathfrak c'$ divides $\mathfrak c$. For every $z\in\mathbb C$ let $I(\psi_E,\mathfrak c,\alpha,z)$ denote the integral:
$$
I(\psi_E,\mathfrak c,\alpha,z)=\int_{F^*\backslash\mathbb A^*}\!\!\!\!
\psi_E(\left(\begin{matrix} y&y\widehat c\\0&1
\end{matrix}\right))\alpha(y)z^{\deg(y)}d\mu^*(y)\in\mathbb C.$$
\end{defn}
\begin{lemma} The integral $I(\psi_E,\mathfrak c,\alpha,z)$ is well-defined and it is independent of the choice of $c$ as the notation indicates.
\end{lemma}
\begin{proof} By Lemma 2 of [28] on pages 300-301 the integrand of $I(\psi_E,\mathfrak c,\alpha,z)$ is compactly supported. Hence the latter is well-defined. Choose another id\`ele $c'\in\mathbb A^*$ such that $(c')=\mathfrak c$. Then there is an $u\in\mathcal O^*$ such that $c'=uc$ and therefore $\widehat c'=u\widehat c$. Now the claim follows from the $GL_2(\mathcal O)$-invariance of $\psi_E$.
\end{proof}
\begin{notn} Let $W(\overline F|F)<\text{Gal}(\overline F|F)$ denote the Weil group of $F$ (as defined in 2.4 of [2] on page 524). Global class field theory furnishes an isomorphism $\mathbf j:F^*\backslash\mathbb A^*\rightarrow W(\overline F|F)^{ab}$ which is compatible with the isomorphism between $F_x^*$ and $W(\overline F_x|F_x)^{ab}$ introduced in Notation 3.2 for every $x\in|\mathcal C|$ (in the sense of 2.4 of [2] on page 525). For every homomorphism $\alpha:\text{Gal}(\overline F|F)\rightarrow\mathbb C^*$ let the same symbol $\alpha$ denote the composition of $\mathbf j$, considered here as an imbedding $\mathbf j:F^*\backslash\mathbb A^*\rightarrow\text{Gal}(\overline F|F)^{ab}$, with the map $\text{Gal}(\overline F|F)^{ab}\rightarrow\mathbb C^*$ induced by the character $\alpha$. Moreover let $\alpha$ also denote the composition of the quotient map $\mathbb A^*\rightarrow F^*\backslash\mathbb A^*$ and the map $\alpha:F^*\backslash\mathbb A^*\rightarrow\mathbb C^*$ introduced above by the usual abuse of notation.
\end{notn}
\begin{notn} For every divisor $\mathfrak m$ on $\mathcal C$ relatively prime to $\mathfrak c$ and for every $\overline{\mathfrak m}\in\mathbb A^*$ such that
$(\overline{\mathfrak m})=\mathfrak m$ the complex number $\alpha(\overline{\mathfrak m})$ is independent of the choice of $\overline{\mathfrak m}$. We let $\alpha(\mathfrak m)$ denote this common value. Let $\sigma_E$ denote the natural $l$-adic representation of Gal$(\overline F|F)$ on the cohomology group $H^1(E_{ \overline F},\mathbb Q_l)$. By definition $L(E,t)=L(\sigma_E,t)$. The twisted $L$-function $L(\sigma_E\otimes\alpha,t)$ is actually a polynomial in the variable $t$ and therefore it can be evaluated at any complex number $t=z$. Finally let
$G(E,\alpha,\mathfrak c,,t)\in\mathbb C[t]$ denote the polynomial:
$$G(E,\alpha,\mathfrak c,t)=\prod_{x\in\underline{\mathfrak c}-\underline{\mathfrak c}'}
\big(-1+a(x)\alpha(x)(qt)^{\deg(x)}-\alpha(x)^2t^{2\deg(x)}\big).$$
\end{notn}
\begin{prop} Assume that $\mathfrak c$ is square-free and it is relatively prime to $\mathfrak{dn}$. Then the following holds:
$$I(\psi_E,\mathfrak c,\alpha,z)=
\alpha(\mathfrak d^2\mathfrak c')\epsilon(\alpha^{-1})
\prod_{x\in\underline{\mathfrak c}}(q^{\deg(x)}-1)^{-1}
\left(\frac{z^2}{q}\right)^{g-1}
\!\!\!\!\!\!\!\!
G(E,\alpha,\mathfrak c,z)L(\sigma_E\otimes\alpha,zq^{-1}).$$
\end{prop}
\begin{proof} According to the Fourier expansion of $\psi_E$ we have:
$$\psi_E(\left(\begin{matrix} y&y\widehat c\\
0&1\end{matrix}\right)) =\sum_{\eta\in F^*}a(\eta y
\mathfrak d^{-1})\tau(\eta y\widehat c),$$
for every $y\in\mathbb A^*$ and the sum on the right is finite. If we interchange
this summation and the integration we get:
\begin{equation}
I(\psi_E,\mathfrak c,\alpha,q^{-s})=
\int_{\mathbb A^*}a(y\mathfrak d^{-1})
\tau(y\widehat c)\alpha(y)q^{-s\deg(y)}d\mu^*(y)
\end{equation}
for every $s\in\mathbb C$. This computation is justified by Lebesgue's convergence theorem if the second integral is absolutely convergent. But the latter holds if Re$(s)>1/2$ as the function $y\mapsto a(y\mathfrak d^{-1})$ has support on $\mathfrak d^{-1}\mathcal O$ and 
$$|a(y\mathfrak d^{-1})\tau(y\widehat c)|=
|a(y\mathfrak d^{-1})|\leq2q^{-1/2\deg(y\mathfrak d^{-1})}$$
by the Weil conjectures. 

Let $\mathbb A_{\mathfrak c}$ and $\mathcal O_{\mathfrak c}$ denote the restricted direct products $\prod'_{x\not\in\underline{\mathfrak c}}F_x$ and $\prod'_{x\not\in\underline{\mathfrak c}}\mathcal O_x$, respectively. Then $\mathbb A_{\mathfrak c}$ is a locally compact topological ring and $\mathcal O_{\mathfrak c}$ is its maximal compact subring. Let $\nu_{\mathfrak c}^*$ be a Haar measure on $\mathbb A_{\mathfrak c}^*$ such that $\nu_{\mathfrak c}^*(\mathcal O_{\mathfrak c}^*)$ is equal to $1$. Let $|\cdot|_x$ be the absolute value on $F_x$ normalised so that $\mu_x(t\mathcal O)=|t|_x$ for every $y\in F_x$. Using Fubini's theorem the integral (3.9.1) can be rewritten as
\begin{equation}
\int_{\mathbb A_{\mathfrak c}^*}a(y\mathfrak d^{-1})\alpha(y)q^{-s\deg(y)}d\nu_{\mathfrak c}^*(y)\cdot\prod_{x\in\underline{\mathfrak c}}
\int_{F_x^*}a(t)
\tau_x(tc_x^{-1})\alpha_x(t)|t|_x^sd\mu_x^*(t).
\end{equation}
The integrand of the first integral of (3.9.2) is invariant under multiplication by $\mathcal O_{\mathfrak c}^*$. Therefore it is equal to
\begin{equation}
\sum_{\substack{\mathfrak m\in\text{Div}(\mathcal C)\\
\underline{\mathfrak m}\cap\underline{\mathfrak c}=\emptyset
}}\!\!\!\!a(\mathfrak m)
\alpha(\mathfrak m\mathfrak d)q^{-s\deg(\mathfrak m\mathfrak d)}=\alpha(\mathfrak d)q^{-s\deg(\mathfrak d)}\cdot\prod_{x\not\in\underline{\mathfrak c}}
L_x(\sigma_E\otimes\alpha,q^{-(s+1)}).
\end{equation}
For every $x\in\underline{\mathfrak c}$ the corresponding term in the product (3.9.2) can be rewritten as
\begin{eqnarray}\ \quad\quad\int_{F_x^*}a(t)
\tau_x(tc_x^{-1})\alpha_x(t)|t|_x^s
d\mu_x^*(t)=&&\\
\sum_{n=0}^{\infty}
a(x^n)\alpha_x(c_x)^n|c_x|^{ns}_x&&\!\!\!\!\!\!\!\!\!\!\!\!\!
\int_{\mathcal O_x^*}
\tau_x(tc_x^{n-1})\alpha_x(t)d\mu_x^*(t)\nonumber
\end{eqnarray}
because $c_x\in F_x^*$ is a uniformizer, and $a(x^n)=0$ if $n<0$. Suppose now that $x$ divides the conductor of $\alpha$. Then the restriction of $\alpha_x$ onto $\mathcal O^*_x$ is a non-trivial character, therefore
\begin{equation}
\int_{\mathcal O_x^*}
\tau_x(tc_x^{n-1})\alpha_x(t)d\mu_x^*(t)
=\left\{\begin{array}{ll}\frac{\alpha_x(c_x)}{q^{\deg(x)}-1}
\epsilon(F_x,\alpha_x^{-1},\tau_x,\mu_x),&\text{if $n=0$,}
\\
0,&\text{otherwise,}\end{array}\right.
\end{equation}
by (3.4.3.2) of [2] on page 528, where we also used the fact that the additive character $\tau_x$ restricted to $\mathcal O_x$ is trivial. Hence we get
\begin{equation}
\int_{F_x^*}a(t)
\tau_x(tc_x^{-1})\alpha_x(t)|t|_x^s
d\mu_x^*(t)=\frac{\alpha_x(c_x)}{q^{\deg(x)}-1}
\epsilon(F_x,\alpha_x^{-1},\tau_x,\mu_x)
\end{equation}
in this case. Otherwise the restriction of $\alpha_x$ onto $\mathcal O^*_v$ is the trivial character therefore the left hand side of (3.9.4) is equal to
\begin{equation}
\frac{-1}{q^{\deg(x)}-1}+\sum_{n=1}^{\infty}a(x^n)\alpha_x(c_x)^n|c_x|^{ns}_x=
L_x(\sigma_E\otimes\alpha,q^{-(s+1)})-\frac{q^{\deg(x)}}{q^{\deg(x)}-1}.\end{equation}
Because both $\sigma_E$ and $\alpha$ are unramified at $x$ we have:
\begin{equation}
L_x(\sigma_E\otimes\alpha,q^{-(s+1)})=\big(1-a(x)\alpha(x)
q^{-s\deg(x)}+\alpha(x)^2q^{-(2s+1)\deg(x)}\big)^{-1}.
\end{equation}
Hence we get that
\begin{eqnarray}
\int_{F_x^*}a(t)
\tau_x(tc_x^{-1})\alpha_x(t)|t|_x^s
d\mu_x^*\!\!\!\!\!\!\!\!&&\!\!\!\!(t)=\\
\frac{1}{q^{\deg(x)}-1}\big(-1+a(x)\alpha(x)q^{(1-s)\deg(x)}
\!\!\!\!\!\!\!\!&&\!\!\!\!-\alpha(x)^2q^{-2s\deg(x)}\big)
\cdot L_x(\sigma_E\otimes\alpha,q^{-(s+1)})\nonumber
\end{eqnarray}
in this case. By Theorem 3.4 we have:
\begin{equation}
\epsilon(\alpha^{-1})=q^{1-g}\alpha^{-1}(\mathfrak d)q^{\deg(\mathfrak d)}
\prod_{x\in\underline{\mathfrak c}'}\epsilon(F_x,\alpha_x^{-1},\tau_x,\mu_x)
\end{equation}
because according to (3.4.3.1) of [2] on page 528 we have:
$$\epsilon(F_x,\alpha_x^{-1},\tau_x,\mu_x)=\alpha_x^{-1}(\mathfrak d_x)
q^{\deg(\mathfrak d_x)}$$
if $\alpha_x$ is unramified, since we assumed that $\mu_x(\mathcal O_x)=1$ and $\mathfrak c$ and $\mathfrak d$ are relatively prime. Combining (3.9.3), (3.9.6), (3.9.9) and $(3.9.10)$ we get:
\begin{eqnarray}\quad\quad
I(\psi_E,\mathfrak c,\alpha,q^{-s})\!\!\!\!&=&\!\!\!\!
\alpha(\mathfrak d^2\mathfrak c')\epsilon(\alpha^{-1})q^{g-1-(s+1)\deg(\mathfrak d)}
\prod_{x\in\underline{\mathfrak c}}(q^{\deg(x)}-1)^{-1}\\
&&\!\!\!\!\cdot G(E,\alpha,\mathfrak c,q^{-s})
L(\sigma_E\otimes\alpha,q^{-(s+1)}),\nonumber\end{eqnarray}
if we also use that $L_x(\sigma_E\otimes\alpha,q^{-(s+1)})=1$ when $x\in\underline{\mathfrak c}'$. Because $\deg(\mathfrak d)=2g-2$ the claim now follows for every complex number $q^{-s}$ such that Re$(s)>1/2$. But both sides of the equation in the proposition above are polynomials in $z$ hence the claim must hold for every complex number as well.
\end{proof}

\section{Lower and upper bounds}

\begin{notn} For every field $K$ let $\overline K$ denote a
separable closure of $K$. Let $v_q: \overline{\mathbb Q}_p^*\rightarrow\mathbb Q$ denote the $p$-adic valuation normalised such that $v_q(q)=1$. Every polynomial
$P(t)\in\overline{\mathbb Q}_p[t]$ can be written in the form:
$$P(t)=at^k\prod_{i=1}^{n-k}(1-\lambda_it),\quad a\in\overline{\mathbb Q}^*_p,
\quad\lambda_i\in\overline{\mathbb Q}_p,$$
where the $\lambda_i$ are the reciprocal roots of $P(t)$. Let $l_q(P(t))\in\mathbb Q$  denote the non-negative number:
$$l_q(P(t))=\sum_{v_q(\lambda_i)\leq1}(1-v_q(\lambda_i)).$$
Let $\mu_{\infty},\mu_{\infty,p}\subset\overline{\mathbb Q}_p^*$ denote the subgroup of roots of unity and roots of unity whose order is prime to $p$, respectively.
\end{notn}
\begin{lemma} With the same notation as above the following holds:
$$\min_{\epsilon\in\mu_{\infty}}(v_q(P({\epsilon q^{-1}})))=v_q(a)-k-l_q(P(t)).$$
Moreover the minimum is attained at all but finitely many $\epsilon\in\mu_{\infty,p}$.
\end{lemma}
\begin{proof} For a fixed $i=1,2,\ldots,k$ and for every $\epsilon\in\mu_{\infty}$ we have:
$$v_q(1-\lambda_i\epsilon q^{-1})=-1+v_q(\lambda_i-\epsilon^{-1}q)\geq\min\{0,v_q(\lambda_i)-1\},$$
so for all but finitely many $\epsilon\in\mu_{\infty,p}$ we have:
$$v_q(1-\lambda_i\epsilon q^{-1})=\min\{0,v_q(\lambda_i)-1\}.$$
Therefore for all but finitely many $\epsilon\in\mu_{\infty,p}$ we have:
$$v_q(P({\epsilon q^{-1}}))=
\min_{\zeta\in\mu_{\infty}}(v_q(P({\zeta q^{-1}})))=v_q(a)-k+
\sum_{i=1}^{\deg(P)}\min\{0,v_q(\lambda_i)-1\}\text{.}$$
\end{proof}
Let $\mathfrak c$ be a square-free effective divisor on $\mathcal C$ which is relatively prime to $\mathfrak{dn}$. Let $\alpha:F\backslash\mathbb A^*\rightarrow\mathbb C^*$ be a continuous character with finite image whose conductor $\mathfrak c'$ divides $\mathfrak c$.
\begin{lemma} The following holds:
$$\min_{\epsilon\in\mu_{\infty}}(v_q(G(E,\alpha,\mathfrak c,\epsilon)))=0.$$
Moreover the minimum is attained at all but finitely many $\epsilon\in\mu_{\infty,p}$.
\end{lemma}
\begin{proof} For every $x\in\underline{\mathfrak c}-\underline{\mathfrak c}'$
we have $q^{\deg(x)}a(x)\in\mathbb Z$. Hence for every $\epsilon\in\mu_{\infty}$
and $x$ as above we have:
$$v_q(-1+a(x)\alpha(x)q^{\deg(x)}\epsilon^{\deg(x)}-\alpha(x)^2\epsilon^{2\deg(x)})
\geq0,$$
and for all but finitely many $\epsilon\in\mu_{\infty,p}$ the left hand side is equal to $0$. The claim follows by taking the product over all
$x\in\underline{\mathfrak c}-\underline{\mathfrak c}'$. $\square$
\end{proof}
\begin{lemma} For every $y\in\mathbb A^*$ and $x\in\mathbb A$ there are $\eta\in F$, $u\in\mathcal O$ and $c\in\mathbb A^*$ such that
$(c)$ is a square-free effective divisor which is relatively prime to $\mathfrak n$ and 
$$x+u+y^{-1}\eta=\widehat c.$$
\end{lemma}
\begin{proof} Let $z\in\mathbb A^*$ be the unique id\`ele such that for every $v\in|\mathcal C|$ we have $z_v=x^{-1}_v$, if $x_v\not\in\mathcal O_v$, and $z_v=1$, otherwise. Then $\mathfrak z=(z)$ is an effective divisor. Let $\mathfrak y$ denote the divisor of $y^{-1}$ and let $\mathfrak b$ be a square-free effective divisor whose degree is at least $2g-1-\deg(\mathfrak y)$ and which is relatively prime to $\mathfrak{nz}$. Let $Z$ denote the closed scheme of $\mathcal C$ whose sheaf of ideals is $\mathcal O_{\mathcal C}(\mathfrak z)^{\vee}\subseteq\mathcal O_{\mathcal C}$ where for every vector bundle $\mathcal F$ on $\mathcal C$ we let $\mathcal F^{\vee}$ denote the dual of $\mathcal F$. We have an exact sequence
$$\CD H^0(\mathcal C,\mathcal O_{\mathcal C}(\mathfrak{yb}))@>>> H^0(\mathcal C,\mathcal O_{\mathcal C}(\mathfrak{zyb}))
@>i_Z>>H^0(Z,\mathcal O_{\mathcal C}(\mathfrak{zyb})|_Z)\endCD$$
where the first map is induced by the inclusion $\mathcal O_{\mathcal C}
(\mathfrak{yb})\subset\mathcal O_{\mathcal C}(\mathfrak{zyb})$ and the second map $i_Z$ is the restriction map. By the Riemann-Roch theorem for curves:
$$\dim_{\mathbb F_q}H^0(\mathcal C,\mathcal F)=2-2g+\deg(\mathcal F)$$
for every line bundle $\mathcal F$ on $\mathcal C$ whose degree is at least $2g-1$. Comparing the dimensions of $H^0(\mathcal C,\mathcal O_{\mathcal C}(\mathfrak{yb}))$ and $H^0(\mathcal C,\mathcal O_{\mathcal C}(\mathfrak{zyb}))$ we get that the image of $i_Z$  has dimension $\deg(\mathfrak z)$ over $\mathbb F_q$. But the dimension of $H^0(Z,\mathcal O_{\mathcal C}(\mathfrak{zyb})|_Z)$ is the same hence $i_Z$ is surjective.

Let $b\in\mathbb A^*$ be an id\`ele such that $(b)=\mathfrak b$ and $b_v=1$ for every $v\not\in\underline{\mathfrak b}$. Recall that for every id\`ele $t\in\mathbb A^*$ we have: $H^0(\mathcal C,\mathcal O_{\mathcal C}(t))=F\cap t^{-1}\mathcal O$. Moreover for every such $t$ we have: $H^0(Z,\mathcal O_{\mathcal C}(t\mathfrak z)|_Z)=(tz)^{-1}\mathcal O/t^{-1}\mathcal O$. Under these identifications $i_Z$ is the composition of the inclusion $F\cap (zb)^{-1}y\mathcal O\rightarrow(zb)^{-1}y\mathcal O$ and the reduction map: $(zb)^{-1}y\mathcal O\rightarrow (zb)^{-1}y\mathcal O/b^{-1}y\mathcal O$. Therefore by the above there is an $\eta\in F$ such that for every $v\in|\mathcal C|$ we have: $(zby^{-1}\eta)_v\in-1+z_v\mathcal O_v$, if $v\in\underline{\mathfrak z}$, and $(zy^{-1}b\eta)_v=(y^{-1}b\eta)_v\in\mathcal O_v$, otherwise. Because $\mathfrak b$ is relatively prime to $\mathfrak z$ we get that $(y^{-1}\eta)_v\in-x_v+\mathcal O_v$, if $v\in\underline{\mathfrak z}$. Let $c\in\mathbb A^*$ be the unique id\`ele such that $c_v=(y\eta^{-1})_v$, if $v\not\in\underline{\mathfrak z}$ and $(y^{-1}\eta)_v\not\in\mathcal O_v$, and $c_v=1$, otherwise. By the above $(c)$ divides $\mathfrak b$ so it is a square-free effective divisor relatively prime to $\mathfrak n$. Let $u=-x-y^{-1}\eta+\widehat c.$ By the above $u\in\mathcal O$ so this choice of $\eta$, $c$ and $u$ satisfies the requirements of the claim.
\end{proof}
Let $\mathbb X(\mathfrak n)$ denote the set of tamely ramified continuous characters $F\backslash\mathbb A^*\rightarrow\mathbb C^*$ with finite image whose conductor is relatively prime to $\mathfrak n$. 
\begin{prop} The following holds:
$$m(E)\geq d\cdot\sup_{\alpha\in\mathbb X(\mathfrak n)}\big(
l_q(L(\sigma_E\otimes\alpha,t))+g-1-v_q(\epsilon(\alpha^{-1}))\big).$$
When $p$ does not divide the order of $\text{\rm Pic}_0(\mathcal C)(\mathbb F_q)$ the two sides above are actually equal.  
\end{prop}
\begin{proof} Fix an isomorphism $ \overline{\mathbb Q}_p\cong\mathbb C$ and let $\alpha\in\mathbb X(\mathfrak n)$ have conductor $\mathfrak c$. Without the loss of generality we may assume that $\mathfrak c$ is relatively prime to $\mathfrak d$ by changing $\tau$, if necessary. According to Proposition 3.9 we have:
\begin{equation}
I(\psi_E,\mathfrak c,\alpha,z)=
\alpha(\mathfrak d^2\mathfrak c')\epsilon(\alpha^{-1})
\prod_{x\in\underline{\mathfrak c}}(q^{\deg(x)}-1)^{-1}
(z^2/q)^{g-1}L(\sigma_E\otimes\alpha,zq^{-1}).
\end{equation}
Let $\mathcal O_{\mathfrak c}<\mathbb A^*$ denote the subgroup $\mathcal O^*$, if $\mathfrak c=1$, and $1+\mathfrak c\mathcal O$, otherwise. The integrand of the integral on the left hand side of (4.5.1) is constant on the cosets of the subgroup $\mathcal U_{\mathfrak c}<F^*\backslash\mathbb A^*$, where $\mathcal U_{\mathfrak c}=(\mathbb F_q^*\cap\mathcal O_c)\backslash\mathcal O_c$. Because $\mathfrak c$ is square-free $p$ does not divide
$|\mathcal O^*/\mathcal O_{\mathfrak c}|$. Hence $\mu^*(\mathcal U_{\mathfrak c})$ is a rational number whose denominator is not divisible by $p$. Therefore $I(\psi_E,\mathfrak c,\alpha,\epsilon)\in p^{-m(E)}\mathbb Z_p[\epsilon]$ for every $\epsilon\in\mu_{\infty}$. In particular $v_q(I(\psi_E,\mathfrak c,\alpha,\epsilon))$ is at least $-m(E)/d$. Note that for every $\alpha\in\mathbb X(\mathfrak n)$ the $L$-function $L(\sigma_E\otimes\alpha,t)\in1+t\overline{\mathbb Q}_p[t]$ hence by Lemma 4.2 there is an $\epsilon\in\mu_{\infty}$ such that $v_q(L(\sigma_E\otimes\alpha,{\epsilon q^{-1}}))=-l_q(L(\sigma_E\otimes\alpha,t))$.
Hence the first part of the claim above is true.

Assume now that $p$ does not divide the order of $\text{\rm Pic}_0(\mathcal C)(\mathbb F_q)$ and let $h(E)$ denote the right hand side of the inequality in the claim above. As we already noted in the proof of Proposition 2.4 in order to show the second half of the claim we only need to show that $p^{h(E)}\psi_E(g)\in \mathbb Z$ for every element $g=\left(\smallmatrix y&yx\\0&1\endsmallmatrix\right)\in P(\mathbb A)$. By Lemma 4.4 there are $\eta\in F$, $u\in\mathcal O$ and $c\in\mathbb A^*$ such that $(c)$ is a square-free effective divisor which is relatively prime to $\mathfrak n$ and $x+u+y^{-1}\eta=\widehat c$.  Because
$$\left(\begin{matrix}1&\eta\\0&1\end{matrix}\right)\cdot
\left(\begin{matrix} y&yx\\0&1\end{matrix}\right)\cdot
\left(\begin{matrix}1&u\\0&1\end{matrix}\right)=
\left(\begin{matrix} y&y(x+u+y^{-1}\eta)\\0&1\end{matrix}\right)$$
and $\psi_E$ is invariant on the left with respect to $P(F)$ and on the right with respect to $P(\mathcal O)$, we may assume that $x=\widehat c$.

Let $\mathfrak c$ be $(c)$ and let $H$ denote the quotient group $\mathbb
A^*/F^*\mathcal O_{\mathfrak c}$ where we continue to use the notation above. Then $H$ can be decomposed as a direct product $G\times\mathbb Z$ where $G$ is a finite subgroup. By class field theory $|G|=|\text{Pic}_0(\mathcal C)|\cdot|\mathcal O^*/(\mathcal O_{\mathfrak c}\mathbb F_q^*)|$. As we already noted above $p\!\!\not||\mathcal O^*/\mathcal O_{\mathfrak c}|$ hence $p$ does not divide the order of $G$ by our assumption. Let $(\cdot)_H:\mathbb A^*\rightarrow H$ be the quotient map and let $[\cdot]:H\rightarrow G$ denote the projection onto the factor $G$. Note
that for every $y\in\mathbb A^*$ the value
$$\psi_E(\left(\begin{matrix} y&y\widehat c\\0&1\end{matrix}\right))t^{\deg(y)}\in\mathbb Q[G][t,t^{-1}]$$
only depends on $(y)_H$. Let $f:H\rightarrow\mathbb Q[t,t^{-1}]$
be the corresponding function and define $I=
\sum_{\substack{g\in G\\n\in\mathbb Z}} c_{g,n}gt^n\in\mathbb Q[G][t,t^{-1}]$ by the formula:
$$I=\sum_{g\in G}\big(\sum_{\substack{h\in H\\ [h]=g}}
 f(h)\big)g.$$
The function $I$ is well-defined because for every $g\in G$ the set of those
$h\in H$ such that $[h]=g$ and $f(y)\neq0$ is finite. Moreover the set
$\{c_{g,n}|g\in G,n\in\mathbb Z\}$ and the image of the function $\mathbb A^*\rightarrow\mathbb
Q$ given by the rule $y\mapsto\psi_E(\left(\smallmatrix y&y\widehat c\\0&1\endsmallmatrix\right))$ are equal hence it will be sufficient to show that $p^{h(E)}I\in\mathbb Z[G][t,t^{-1}]$. For every group homomorphism $\alpha:G\rightarrow\overline{\mathbb Q}_p^*$ by slight abuse of notation let the same symbol $\alpha$ denote the unique ring homomorphism $\mathbb Q[G][t,t^{-1}]\rightarrow\overline{\mathbb Q}_p[t,t^{-1}]$ whose restriction onto $G$ is $\alpha$ and $\alpha(t)=t$. Because $p$ does not divide $|G|$ it will be sufficient to show that $v_q(a)\leq-h(E)/d$ for every coefficient $a$ of $\alpha(I)$ for every $\alpha$ as above. Let the symbol $\alpha$ denote also the composition $\alpha\circ[\cdot]\circ(\cdot)_H:\mathbb A^*\rightarrow\overline{\mathbb Q}_p^*$. Choose the character $\tau$ so that $\mathfrak d$ is relatively prime to $\mathfrak c$. Then $\alpha(I)$ is the polynomial:
$$\alpha(I)=
\alpha(\mathfrak d^2\mathfrak c')\epsilon(\alpha^{-1})
\prod_{x\in\underline{\mathfrak c}}(q^{\deg(x)}-1)^{-1}
\left(\frac{t^2}{q}\right)^{g-1}
G(E,\alpha,\mathfrak c,t)L(\sigma_E\otimes\alpha,tq^{-1})$$
by Proposition 3.9. Because for every polynomial $P(t)=\sum_{k=0}^Nb_kt^k\in\overline{\mathbb Q}_p$we have:
$$\min_{\epsilon\in\mu_{\infty}}(v_q(P(\epsilon)))=\min_{0\leq k\leq N}(v_q(b_k)),$$
the claim now follows from Lemmas 4.2 and 4.3.
\end{proof}

\section{Galois module structure of the coherent cohomology of elliptic surfaces}

\begin{defn} In this chapter $G$ will be a finite group. Let $A$ be a Noetherian ring and let $Y$ be a scheme which is separated and of finite type over Spec$(A)$. By an $A[G]$-module on $Y$ we mean a sheaf of $A[G]$-modules on $Y$. These form the objects of a category whose morphisms are maps respecting the $A[G]$-module structure. Suppose now that $\mathcal F$ is an $A[G]$-module on $Y$ which is also an $\mathcal O_Y$-module in such a way that the actions of $\mathcal O_Y$ and $G$ commute and the $A[G]$-module structure of $\mathcal F$ respects the structure morphism $Y\rightarrow\text{Spec}(A)$. If $\mathcal F$ is also a quasi-coherent (resp. coherent) $\mathcal O_Y$-module, then we will call $\mathcal F$ a quasi-coherent (resp. coherent) $\mathcal O_Y[G]$-module. As noted in [1] on page 447 there are enough injectives in the category of $A[G]$-modules on $Y$ hence the global section functor $\Gamma$ has a derived functor into the localisation of the category of complexes of $A[G]$-complexes bounded from below with respect to the multiplicative system of quasi-isomorphisms which will be denoted by $R\Gamma^+$. 
\end{defn}
\begin{defn} An $A[G]$-module $M$ is cohomologically trivial if the Tate cohomology group $\widehat H^i(H,M)$ vanishes for all subgroups $H$ of $G$ and for all integers $i$. Let $CT(A[G])$ denote the Grothendieck group of all finitely generated $A[G]$-modules which are cohomologically trivial. For every cohomologically trivial $A[G]$-module $M$ let $[M]$ denote its class in $CT(A[G]))$. Suppose that $\mathcal F$ is a quasi-coherent $\mathcal O_Y$-module such that each stalk of $\mathcal F$ is a cohomologically trivial $A[G]$-module. Then by Theorem 1.1 of [1] on page 447 the complex $R\Gamma^+(\mathcal F)$ is isomorphic in the derived category of $A[G]$-modules to a bounded complex $M^*$ of finitely generated cohomologically trivial $A[G]$-modules. Moreover the Euler characteristic $\sum(-1)^i[M^i]$ in $CT(A[G])$ only depends on $\mathcal F$ and will be denoted by $\chi(\mathcal F)$.
\end{defn}
\begin{defn} Let $X$ be a normal scheme which is of finite type over Spec$(A)$. Assume that the finite group $G$ acts on $X$ on the left. Let $\mathcal F$ be a coherent sheaf on $X$. A $G$-linearisation on $\mathcal F$ is a collection $\Psi=\{\psi_g\}_{g\in G}$ of isomorphisms $\psi_g:g_*(\mathcal F)\rightarrow\mathcal F$ for every $g\in G$ such that
\begin{enumerate}
\item[$(i)$] we have $\psi_1=\text{Id}_{\mathcal F}$,
\item[$(ii)$] for every $g$, $h\in G$ we have $\psi_{hg}=\psi_h\circ
h_*(\psi_g)$,
\end{enumerate}
where $h_*(\psi_g):(hg)_*(\mathcal F)=h_*(g_*(\mathcal F))\rightarrow h_*(\mathcal F)$ is the direct image of the map $\psi_g:g_*(\mathcal F)\rightarrow\mathcal F$ under the action of $h$. We define a $G$-sheaf over $X$ to be a sheaf on $X$ equipped with a $G$-linearisation. A coherent $G$-sheaf is a coherent sheaf on $X$ equipped with a $G$-linearisation $\Psi$ such that $\psi_g:g_*(\mathcal F)\rightarrow\mathcal F$ is $\mathcal O_X$-linear for every $g\in G$. 
\end{defn}
\begin{defn} Let $f:X\rightarrow Y$ be a tame $G$-cover as defined in Definition 2.2 of [1] on page 451 and let $\mathcal F$ be a coherent $G$-sheaf on $X$. The $G$-linearisation on $\mathcal F$ induces an $\mathcal O_Y$-linear action of $G$ on the direct image sheaf $f_*(\mathcal F)$ which makes the latter a coherent $\mathcal O_Y[G]$-module. By Theorem 2.7 of [1] on page 452 each stalk of the $\mathcal O_Y[G]$-sheaf $f_*(\mathcal F)$ is a cohomologically trivial $A[G]$-module. Hence the Euler characteristic $\chi(f_*(\mathcal F))\in CT(A[G])$ introduced in Definition 5.2 is well-defined. The later will be denoted by 
$\chi(G,\mathcal F)$.
\end{defn}
Suppose now that $A$ is a field and its characteristic does not divide the order of $G$. Then every finitely generated $A[G]$-module is cohomologically trivial. Also assume that $Y$ is proper over $\text{\rm Spec}(A)$ and let $\mathcal F$ be again a coherent $G$-sheaf on $X$. By the above for every $n\in\mathbb N$ the cohomology group $H^n(X,\mathcal F)$ is a cohomologically trivial, finitely generated $A[G]$-module with respect to the natural $A[G]$-action.
\begin{lemma} With these assumptions we have
$$\chi(G,\mathcal F)=\sum_{n\in\mathbb N}(-1)^n
[H^n(X,\mathcal F)]\in CT(A[G]).$$
\end{lemma}
\begin{proof} Because finite maps are affine the higher derived sheaves $R^if_*(\mathcal F)$ are vanishing. Hence $H^n(X,\mathcal F)=H^n(Y,f_*(\mathcal F))$ as $A[G]$-modules. Now the claim follows from Proposition 1.5 of [1] on page 449.
\end{proof}
In addition to the assumptions above also suppose now that $f:X\rightarrow Y$ above is a map of smooth, projective curves over Spec$(A)$. Let $\mathcal L$ be a line bundle on $Y$. The line bundle $f^*(\mathcal L)$ on $X$ is naturally equipped with the structure of a coherent $G$-sheaf. 
\begin{lemma} With the same notation and assumptions as above the following equation holds in $CT(A[G])$:
$$\chi(G,f^*(\mathcal L))=\chi(G,\mathcal O_X)+\deg(\mathcal L)[A[G]].$$
\end{lemma}
\begin{proof} Of course we are going to show the claim with the usual devisage argument. By the Riemann-Roch theorem there is a divisor $D$ on $Y$ whose support is disjoint from the ramification divisor of the cover $f$ and $\mathcal L=O_Y(D)$. First assume that $D$ is effective. When $\deg(D)=0$ the claim is obvious. Otherwise $D=D'+\mathfrak p$ where $D'$ is an effective divisor with $\deg(D')<\deg(D)$ and $\mathfrak p$ is a closed point on $Y$. There is a short exact sequence:
$$\CD0@>>> f^*(\mathcal O_Y(D'))@>>> f^*(\mathcal O_Y(D))
@>>> f^*(A{\mathfrak p})@>>> 0\endCD$$
where $A{\mathfrak p}$ denotes the skyskaper sheaf on $Y$ with support $\mathfrak p$. Because $\mathfrak p$ is not in the ramification locus of $f$ we have $H^0(f_*f^*(A{\mathfrak p}),Y)\cong A[G]^{\deg(\mathfrak p)}$ as $A[G]$-modules. Moreover all higher cohomology groups of the skyskaper sheaf $f_*f^*(A{\mathfrak p})$ vanish. Hence by the additivity of the Euler characteristic we get:
$$\chi(G,f^*(\mathcal O_Y(D)))=\chi(G,f^*(\mathcal O_Y(D')))+\deg(\mathfrak p)[A[G]].$$
Now the claim follows by induction on $\deg(D)$. Consider next the general case and write $D=D_1-D_2$ where $D_1$ and $D_2$ are divisors on $X$ whose supports are disjoint. We are going to prove the claim by induction on $\deg(D_2)$. We already proved the claim when $\deg(D_2)=0$. Otherwise $D_2=D'_2+\mathfrak p$ where $D'_2$ is an effective divisor with $\deg(D'_2)<\deg(D_2)$ and $\mathfrak p$ is a closed point on $Y$. By repeating the same argument which we used above we get:
$$\chi(G,f^*(\mathcal O_Y(D)))=\chi(G,f^*(\mathcal O_Y(D_1-D_2')))-\deg(\mathfrak p)[A[G]].$$
The claim is now clear.
\end{proof}
\begin{defn} Assume now that $A=\mathbf k$ is a perfect field of characteristic $p$ and let $W$ denote the ring of Witt vectors of $\mathbf k$ of infinite length. Moreover let $K$ denote the field of fractions of $W$. Let $M$ be a finitely generated cohomologically trivial $\mathbf k[G]$-mo\-dule. Then $M$ is a projective $\mathbf k[G]$-module by claim $(a)$ of Proposition 4.1 of [1] on page 557. Hence $M$ is isomorphic to $P/pP$ for some finitely generated projective $W$-module $P$. The character of the $\overline K[G]$-module $\overline K\otimes_WP$ can be written in the form $\sum m_{\alpha}\alpha$ where the sum is over the set $R(G)$ of irreducible $\overline K$-valued characters of $G$. The integer $m_{\alpha}$ is independent of the choice of $P$. Let $\Delta(M):R(G)\rightarrow\mathbb Z$ be the function defined by the formula $\Delta(M)(\alpha)=m_{\alpha}$. This map extends uniquely to a homomorphism from $CT(A[G])$ to the group of $\mathbb Z$-valued functions on $R(G)$ which will be denoted by $\Delta$ as well.
\end{defn}
Suppose that $A$ is a finite extension of the field $B$. Then the restriction of operators from $A[G]$ to $B[G]$ induces a homomorphism $\text{Res}_{A\rightarrow B}:CT(A[G])\rightarrow CT(B[G]))$. Assume now that $A$ is the finite field $\mathbb F_q$. Then every
$\alpha\in R(G)$ can be considered as a representation of the absolute Galois group of the function field of $X$. In particular the $\epsilon$-constant $\epsilon(\alpha)\in\overline{\mathbb Q}_p$ is defined. 
\begin{thm} With the same notation and assumptions as above we have:
$$\Delta(\text{\rm Res}_{\mathbb F_q\rightarrow\mathbb F_p}(\chi(G,\mathcal O_X)))
(\alpha)=-dv_q(\epsilon(\alpha^{-1}))\quad(\forall\alpha\in R(G)).$$
\end{thm}
\begin{proof} This is a special case of Theorem 5.2 of [1] on pages 463-464.
\end{proof}
\begin{notn} Suppose now that $Y=\mathcal C$ and let $\mathcal E$ and $\Delta_E$ be the same as in the introduction. Assume that the ramification divisor of the cover $f:X\rightarrow Y$ has support disjoint from the conductor of $E$. Let $g':\mathcal E'\rightarrow X$ be the base change of the elliptic fibration $g:\mathcal E\rightarrow Y$ with respect to the map $f$. Note that the $X$-scheme $\mathcal E'$ is a relatively minimal regular model of the base change of $E$ to the function field of $X$. Moreover $\mathcal E'$ is equipped with a unique action of $G$ fixing the zero section such that $g'$ is equivariant with respect to this action and the one on $X$.  
\end{notn}
\begin{thm} The $\mathbb F_q[G]$-module $H^2(\mathcal E',\mathcal O_{\mathcal E'})$ is cohomologically trivial and 
$$\Delta(\text{\rm Res}_{\mathbb F_q\rightarrow\mathbb F_p}([H^2(\mathcal E',\mathcal O_{\mathcal E'})])
(\alpha)=\frac{d}{12}\deg(\Delta_E)+dv_q(\epsilon(\alpha^{-1}))\ \ (\forall\alpha\in R(G)).$$
\end{thm}
\begin{proof} By Lemma 5.5 we have:
\begin{equation}
\chi(G,\mathcal O_{\mathcal E'})=[H^0(\mathcal E',\mathcal O_{\mathcal E'})]-
[H^1(\mathcal E',\mathcal O_{\mathcal E'})]+[H^2(\mathcal E',\mathcal O_{\mathcal E'})].
\end{equation}
By Lemma 4 of [6] on page 79 the map $(g')^*:\text{Pic}^0(X)\rightarrow\text{Pic}^0(\mathcal E')$ induced by Picard functoriality is an isomorphism. Because this map is equivariant with respect to the induced $G$-actions on $\text{Pic}^0(X)$ and $\text{Pic}^0(\mathcal E')$, we get that $H^1(\mathcal E',\mathcal O_{\mathcal E'})=H^1(X,\mathcal O_X)$ as $\mathbb F_q[G]$-modules, since these modules are isomorphic to the tangent spaces at the zero of the abelian varieties $\text{Pic}^0(\mathcal E')$ and $\text{Pic}^0(X)$, respectively. Obviously
$H^0(\mathcal E',\mathcal O_{\mathcal E'})=H^0(X,\mathcal O_X)$ as $\mathbb F_q[G]$-modules,
hence from (5.10.1) and Lemma 5.5 we get that
\begin{equation}
[H^2(\mathcal E',\mathcal O_{\mathcal E'})]=\chi(G,\mathcal O_{\mathcal E'})-
\chi(G,\mathcal O_X).
\end{equation}
Let $\Omega^1_{\mathcal E/Y}$, $\Omega^1_{\mathcal E'/X}$ denote the sheaf of relative K\"ahler differentials of the $Y$-scheme $\mathcal E$ and the $X$-scheme $\mathcal E'$, respectively. Let $\omega_{\mathcal E/Y}$, $\omega_{\mathcal E'/X}$ denote respectively the pull-back of $\Omega^1_{\mathcal E/Y}$ and $\Omega^1_{\mathcal E'/X}$ with respect to the zero section. These sheaves are line bundles on $Y$ and $X$, respectively. Moreover by Grothendieck's duality we have $R^1g'_*(\mathcal O_{\mathcal E'})=\omega^{\otimes-1}_{\mathcal E'/X}$. In particular $\chi(G,R^1g'_*(\mathcal O_{\mathcal E'}))=-\chi(G,\omega_{\mathcal E'/X})$. Because all boundary maps in the spectral sequence  $H^p(Y,R^qg'_*(\mathcal O_{\mathcal E'}))
\Rightarrow H^{p+q}(\mathcal E',\mathcal O_{\mathcal E'})$ are $\mathbb F_q[G]$-linear we get from the above and Lemma 5.5 that
\begin{equation}
\chi(G,\mathcal O_{\mathcal E'})=\chi(G,\mathcal O_X)-
\chi(G,R^1g'_*(\mathcal O_{\mathcal E'}))=\chi(G,\mathcal O_X)-
\chi(G,\omega^{\otimes-1}_{\mathcal E'/X}).
\end{equation}
Combining (5.10.2) and (5.10.3) we get that
\begin{equation}
[H^2(\mathcal E',\mathcal O_{\mathcal E'})]=-\chi(G,\omega^{\otimes-1}_{\mathcal E'/X}).
\end{equation}
By definition $\Delta_E$ is the zero divisor of a non-zero section of $\omega_{\mathcal E/Y}^{ \otimes12}$. Therefore $\deg(\Delta_E)=12\deg(\omega_{\mathcal E/Y})$. Moreover $\omega_{\mathcal E'/X}=f^*(\omega_{\mathcal E/Y})$. Hence (5.10.4) and Lemma 5.6 implies that
$$[H^2(\mathcal E,\mathcal O_{\mathcal E})]=
\frac{\deg(\Delta_E)}{12}[\mathbb F_q[G]]-\chi(G,\mathcal O_X).$$
The claim now follows from Theorem 5.8.
\end{proof}

\section{Slope estimates}

\begin{defn} Let $\sigma:W\rightarrow W$ denote the absolute Frobenius automorphism. Let $W((V))$ denote the $W$-algebra of formal Laurent series $\sum_{i\geq n}a_iV^i$, where $a_i\in W$ and $n\in\mathbb Z$ are arbitrary, with the usual addition and with multiplication defined by the formula:
$$\big(\sum_{i\geq n}a_iV^i\big)\cdot\big(\sum_{j\geq m}b_jV^j\big)
=\sum_{k\geq n+m}(\sum_{i+j=k}a_i\sigma^{-i}(b_j))V^k.$$
Moreover let $W[V]$ and $W[[V]]$ denote the subring of $W((V))$ consisting of polynomials and formal power series in the variable $V$, respectively. Let $M$ be a module over the ring $W[V]$. Then the kernel and the cokernel of the multiplication by $V:M\rightarrow M$ are $W$-modules and we define
$$\chi(M)=\text{length}_W(\text{Ker}(V))-
\text{length}_W(\text{Coker}(V))$$
provided both numbers on the right are finite. Let $W[V,F]$ denote ring generated by the variable $V$ over $W$ subject to the relations $VF=FV=p$, $Fc=\sigma(c)F$ and $Vc=\sigma^{-1}(c)V$ for every $c\in W$. For every module $M$ over the Dieudonn\'e ring $W[F,V]$ which is free and finitely generated as a $W$-module the tensor product $M\otimes_WK$ is an $F$-isocrystal over $W$ with respect to multiplication by $F\otimes_W\text{id}_K$. For every $F$-isocrystal $M$ over $W$ let $l(M)$ denote $\sum m_i(1-\lambda_i)$ where the $\lambda_i$ are the slopes of $M$ and $m_i$ is the multiplicity of $\lambda_i$.
\end{defn}
\begin{lemma} Then the following hold:
\begin{enumerate}
\item[$(i)$] if $0\rightarrow M'\rightarrow M\rightarrow M\rightarrow0$ is a short exact sequence of $W[V]$-modules, then $\chi(M)$ is defined if $\chi(M')$ and $\chi(M'')$ are defined and
$$\chi(M)=\chi(M')+\chi(M''),$$
\item[$(ii)$] if $M$ is a module over the ring $W[F,V]$ which is free and finitely generated as a $W$-module, then $\chi(M)=-l(M\otimes_WK)$,
\item[$(iii)$] if $M$ is a finitely generated torsion $W[[V]]$-module then 
$$\chi(M)=-\text{\rm length}_{W((V))}M\otimes_{W[[V]]}W((V)).$$
\end{enumerate}
\end{lemma}
\begin{proof} This is Lemma 7.2 of [17] on page 530.
\end{proof}
\begin{defn} Let $G$ be a finite abelian group whose order is not divisible by $p$ and assume that $\mathbf k$ is algebraically closed. Then every $\alpha\in R(G)$ takes values in $W$ hence the element $\pi_{\alpha}=\sum_{g\in G}\alpha(g)^{-1}g\in W[G]$ is well-defined. For every $W[G]$-module $M$ let $M^{\alpha}$ denote the sub $W[G]$-module $\pi_{\alpha}M$. Finally for every $F$-isocrystal $M$ over $W$ and real numbers $a\leq b$ let $M_{[a,b)}$ denote the maximal subquotient of $M$ with slopes in the interval $[a,b)$.
\end{defn}
Let $X$ be a smooth projective variety defined over $\mathbf k$ and assume that $G$ acts on $X$. Let $H^m(X/W)$ denote the $m$-th crystalline cohomology group of $W$. Then $W[G]$ acts on $H^m(X/W)$ and $H^m(X,\mathcal O_X)$ by functoriality for every $m$.  
\begin{prop} We have:
$$l(H^r(X/W)^{\alpha}\otimes_WK_{[0,1)})
\leq\dim_{\mathbf k}(H^r(X,\mathcal O_X)^{\alpha})$$
for every $r\in\mathbb N$ and $\alpha\in R(G)$. 
\end{prop}
\begin{proof} Let $H^r(X,W)$ denote the $r$-th Witt vector cohomology group of the variety $X$. The latter is a $W[V]$-module and there is a long exact sequence:
$$\cdots\rightarrow H^{r-1}(X,\mathcal O_X)\rightarrow H^r(X,W)
\rightarrow^{\!\!\!\!\!\!V}
H^r(X,W)\rightarrow H^r(X,\mathcal O_X)\rightarrow\cdots$$
The group $G$ acts on all cohomology groups in the sequence above and the maps are equivariant with respect to this action so there is a long exact sequence:
$$\cdots\rightarrow H^{r-1}(X,\mathcal O_X)^{\alpha}\rightarrow H^r(X,W)^{\alpha}
\rightarrow^{\!\!\!\!\!\!V}
H^r(X,W)^{\alpha}\rightarrow H^r(X,\mathcal O_X)^{\alpha}\rightarrow\cdots$$
Because $X$ is projective the vector spaces $H^r(X,\mathcal O_X)^{\alpha}$ have finite dimension. Therefore $\chi(H^r(X,W)^{\alpha})$ is well-defined and
\begin{equation}
\chi(H^r(X,W)^{\alpha})\geq-
\dim_{\mathbf k}(H^r(X,\mathcal O_X)^{\alpha}).
\end{equation}
Write
$$H^r(X,W)^{\alpha}_t=\text{Ker}(H^r(X,W)^{\alpha}
\rightarrow H^r(X,W)^{\alpha}\otimes_WK).$$
Then $H^r(X,W)^{\alpha}_t$ is a torsion $W[[V]]$-module and there is a short exact sequence of $W[V]$-modules:
$$0\rightarrow H^r(X,W)^{\alpha}_t\rightarrow H^r(X,W)^{\alpha}\rightarrow H^r(X,W)^{\alpha}_{ct}\rightarrow0$$
such that $H^r(X,W)^{\alpha}_{ct}$ is a module over the Dieudonn\'e ring $W[F,V]$ which is free and finitely generated as a $W$-module. Therefore by part $(i)$ of Lemma 6.2 we have:
\begin{equation}
\chi(H^r(X,W)^{\alpha})=\chi(H^r(X,W)_t^{\alpha})+\chi(H^r(X,W)_{ct}^{\alpha}).
\end{equation}
Because the slope spectral sequence degenerates modulo torsion (Th\'eor\`eme 3.2 of [8] on page 615) there is an isomorphism between the $F$-isocrystals $H^r(X,W)\otimes_WK$ and $H^r(X/W)\otimes K_{[0,1)}$. Since this isomorphism is equivariant with respect to the action of $G$ we get that the $F$-isocrystals $H^r(X,W)^{\alpha}\otimes_WK$ and $H^r(X/W)^{\alpha}\otimes_WK_{[0,1)}$ are also isomorphic. Hence
\begin{equation}
\chi(H^r(X,W)_{ct}^{\alpha})=
-l(H^r(X,W)_{ct}^{\alpha}\otimes_WK)=
-l(H^r(X/W)^{\alpha}\otimes_WK_{[0,1)}).
\end{equation}
by part $(ii)$ of Lemma 6.2. According to  part $(iii)$ of Lemma 6.2 the number $\chi(H^r(X,W)_t^{\alpha})$ is not positive. Hence (6.4.1), (6.4.2) and (6.4.3) imply that
$$-l(H^r(X/W)^{\alpha}\otimes_WK_{[0,1)})\geq
-\dim_{\mathbf k}(H^r(X,\mathcal O_X)^{\alpha})\text{.\qedhere}$$
\end{proof}
\begin{notn} Assume now that $\mathbf k$ is the algebraic closure of $\mathbb F_q$ and for every Spec$(\mathbb F_q)$-scheme $S$ let $\overline S$ denote its base change to Spec$(\mathbf k)$. Moreover let $F:\overline S\rightarrow\overline S$ denote the Frobenius relative to $\mathbb F_q$ for every such $S$. Assume now that $X$ is a smooth projective variety defined over $\mathbf k$ equipped with an action of $G$. Recall that we've chosen a prime $l\neq p$. Now fix an isomorphism $\nu:\overline{\mathbb Q}_l\rightarrow\overline{\mathbb Q}_p$. Since group $G$ acts on the \'etale cohomology group $H^m(\overline X,\overline{\mathbb Q}_l)$ we may consider the latter as a $\mathbb W[G]$-module if we identify $\overline{\mathbb Q}_l$ and $\overline{\mathbb Q}_p$ via $\nu$. Then the map $F^*:H^m(\overline X,\overline{\mathbb Q}_l)\rightarrow H^m(\overline X,\overline{\mathbb Q}_l)$ induced by the Frobenius morphism $F$ commutes with the action of $G$ so $H^m(\overline X,\overline{\mathbb Q}_l)^{\alpha}$ is an $F^*$-invariant subspace. 
\end{notn}
\begin{lemma} We have:
$$l_q(\det(1-F^*t|H^m(\overline X,\overline{\mathbb Q}_l)^{\alpha}))=
l(H^m(\overline X/W)^{\alpha}\otimes_WK_{[0,1)}).$$
\end{lemma}
\begin{proof} For every $g\in G$ let the same symbol denote the base
change $\overline X\rightarrow\overline X$ to Speck$(\mathbf k)$ of the automorphism $X\rightarrow X$ furnished by the given action of $G$ on $X$. For every positive integer $n$ the composition
$$g\circ F^{[n]}=g\circ\underbrace{F\circ\cdots\circ F}_{\text{$n$-times}}$$
induces the zero map on tangent spaces, hence the Lefschetz trace formula applied to $g\circ F^{[n]}$ both in the $l$-adic and the crystalline cohomology theories imply together that
\begin{eqnarray}
\sum_{m=0}^{2\dim(\overline X)}(-1)^m
\text{Tr}(1-\!\!\!\!\!\!\!\!&&\!\!\!\!g^*
(F^*)^n|H^m(\overline X,\overline{\mathbb Q}_l))=\\
\!\!\!\!\!\!\!\!&&\!\!\!\!\sum_{m=0}^{2\dim(\overline X)}
(-1)^m
\text{Tr}(1-g^*(F^*)^n|H^m(\overline X/W)\otimes_WK).\nonumber
\end{eqnarray}
Let $l_m(\alpha,t)$ and $c_m(\alpha,t)$ denote the polynomials:
$$\det(1-F^*t|H^m(\overline X,\overline{\mathbb Q}_l)^{\alpha})
\text{ and }
\det(1-F^*t|H^m(\overline X/W)\otimes_WK^{\alpha}),$$
respectively. By the orthogonality of characters the equation (6.6.1) implies that
\begin{equation}
\prod_{m=0}^{2\dim(\overline X)}l_m(\alpha,t)^{(-1)^m}=
\prod_{m=0}^{2\dim(\overline X)}c_m(\alpha,t)^{(-1)^m}.
\end{equation}
By Deligne's purity theorem and Theorem 1 of [11] on page 74 the reciprocal roots of the polynomials $l_m(\alpha,t)$ and $c_m(\alpha,t)$ are Weil numbers with squared complex norm $q^m$. Hence there are no cancellations in the alternating products in (6.6.2) above. Therefore $l_m(\alpha,t)=c_m(\alpha,t)$ for every $m$. Let $\lambda_1,\lambda_2,\ldots,\lambda_k$ be the eigenvalues of $F^*$ considered as a linear transformation of $H^m(\overline X/W)\otimes_WK^{\alpha}$. The claim now follows from the fact that $v_q(\lambda_1),v_q(\lambda_2),\ldots,v_q(\lambda_k)$ are the slopes of the $F$-isocrystal
$H^m(\overline X/W)\otimes_WK^{\alpha}$.
\end{proof}
\begin{notn} Let us consider now the same situation as in the introduction. Fix a tamely ramified character $\alpha\in\mathbb X(\mathfrak n)$ and let $\pi:X\rightarrow\mathcal C$ be the Galois cover corresponding to the extension $F'|F$ where $F'$ is the subfield of $\overline F$ fixed by the kernel of $\alpha$. Let $G$ denote the Galois group of the cover $\pi$ and let $g':\mathcal E'\rightarrow X$ denote the base change of $g:\mathcal E\rightarrow\mathcal C$ with respect to the map $f$ as in Notation 5.9. We will also keep on using the notation introduced in 6.5.
\end{notn}
\begin{lemma} We have:
$$l_q(L(\sigma_E\otimes\alpha,t))=
l(H^2(\overline{\mathcal E}'/W)^{\alpha}\otimes_WK_{[0,1)}).$$
\end{lemma}
\begin{proof} For every morphism $m:R\rightarrow S$ of
Spec$(\mathbb F_q)$-schemes let $\overline m:\overline R\rightarrow\overline S$ denote the base change to Spec$(\mathbf k)$. The Leray spectral sequence $H^p(\overline X,R^q\overline g'_*(\overline{\mathbb Q}_l(1)))\Rightarrow
H^{p+q}(\overline{\mathcal E'},\overline{\mathbb Q}_l(1))$ furnishes an injection $\zeta:H^1(\overline X,R^1\overline g'_*(\overline{\mathbb Q}_l(1)))\rightarrow H^2(\overline{\mathcal E}',\overline{\mathbb Q}_l(1))$ and the image of this map is the orthogonal complement of the $\overline{\mathbb Q}_l$-linear subspace $V$ spanned by the Chern classes of the zero section and the fibres of the elliptic fibration $\overline{\mathcal E}'\rightarrow\overline{\mathcal C}$, considered here as divisors on the surface $\overline{\mathcal E}'$, with respect to the cup product pairing. The cohomology group $H^1(\overline X,R^1\overline g'_*(\overline{\mathbb Q}_l(1)))$ is naturally equipped with a $G$-action because the map $\overline g'$ is equivariant with respect to the action of $G$ on $\overline{\mathcal E}'$ and $\overline X$, respectively. Moreover $\zeta$ is $G$-linear. Hence we have an isomorphism:
\begin{equation}
H^1(\overline X,R^1\overline g'_*(\overline{\mathbb Q}_l))^{\alpha}\oplus
V^{\alpha}(-1)\cong
H^2(\overline{\mathcal E}',\overline{\mathbb Q}_l)^{\alpha}.
\end{equation}
Note that the cohomology group $H^1(\overline X,R^1\overline g'_*(\overline{\mathbb Q}_l))$ is equipped with the action of a Frobenius operator $F^*$ which commutes with the action of $G$ and the isomorphism in (6.8.1) respects the action of the operator $F^*$ on both sides. Moreover for every eigenvalue $\lambda$ of $F^*$ on $V(-1)$ we have $v_q(\lambda)=1$, hence
$$l_q(\det(1-F^*t|H^1(\overline X,R^1\overline g'_*(\overline{\mathbb Q}_l))^{\alpha}))=l_q(
\det(1-F^*t|H^2(\overline{\mathcal E}',\overline{\mathbb Q}_l)^{\alpha})).$$
We may consider $\alpha$ a lisse $l$-adic sheaf on Spec$(F)$. By slight abuse of notation let $\alpha$ also denote the pull-back onto $\overline{\mathcal C}$ of the direct image of this sheaf $\alpha$ with respect to the open immersion Spec$(F)\rightarrow\mathcal C$. Because the Galois representation $H^1(E_{\overline F},\overline{\mathbb Q}_l)$ is absolutely irreducible and self-dual we have
$$H^0(\overline{\mathcal C},R^1\overline g_*(\overline{\mathbb Q}_l)\otimes\alpha)=
H^2(\overline{\mathcal C},R^1\overline g_*(\overline{\mathbb Q}_l)
\otimes\alpha)=0.$$
Moreover $H^1(\overline{\mathcal C},R^1\overline g_*(\overline{\mathbb Q}_l)\otimes\alpha)=
H^1(\overline X,R^1\overline g'_*(\overline{\mathbb Q}_l))^{\alpha}$ by the degeneration of the Hoch\-schild-Serre spectral sequence. Hence by the Grothendieck-Verdier trace formula:
$$l_q(L(\sigma_E\otimes\alpha,t))=
l_q(\det(1-F^*t|H^1(\overline X,R^1\overline g'_*(\overline{\mathbb Q}_l))^{\alpha})=
l_q(\det(1-F^*t|H^2(\overline{\mathcal E}',\overline{\mathbb Q}_l)^{\alpha})).$$
The claim now follows from Lemma 6.6.
\end{proof}
\begin{thm} Assume that $p$ does not divide the order of $\text{\rm Pic}_0(\mathcal C)(\mathbb F_q)$. Then we have:
$$m(E)\leq d(\frac{1}{12}\deg(\Delta_E)+g-1).$$
\end{thm}
\begin{proof} Let $\alpha\in\mathbb X(\mathfrak n)$ be arbitrary and let $C(\alpha)$ denote the set of all characters $\beta:\text{Gal}(F'|F)
\rightarrow\overline{\mathbb Q}_p^*$ which are conjugate to $\alpha$ under the
action of the absolute Galois group Gal$(\overline{\mathbb Q}_p|\mathbb Q_p)$ on
$\overline{\mathbb Q}_p$. For every $\beta\in C(\alpha)$ we have:
\begin{equation}
l_q(L(\sigma_E\otimes\beta,t))\leq
\dim_{\mathbf k}(H^2(\overline{\mathcal E}',
\mathcal O_{\overline{\mathcal E}'})^{\beta})=
\dim_{\mathbf k}(H^2(\mathcal E',\mathcal O_{\mathcal E'})
\otimes_{\mathbb F_q}\mathbf k^{\beta})
\end{equation}
by Proposition 6.4 and Lemma 6.8. According to Remark 4.5 of [1] on page 460 we have 
\begin{equation}
\Delta(\text{Res}_{\mathbb F_q\rightarrow\mathbb F_p}
([H^2(\mathcal E',\mathcal O_{\mathcal E'})])(\alpha)=
\frac{d}{|C(\alpha)|}\sum_{\beta\in C(\alpha)}
\dim_{\mathbf k}(H^2(\mathcal E',\mathcal O_{\mathcal E'})
\otimes_{\mathbb F_q}\mathbf k^{\beta}).
\end{equation}
Because the action of  Gal$(\overline{\mathbb Q}_p|\mathbb Q_p)$ on $\overline{\mathbb Q}_p$ leaves the valuation $v_q$ invariant we have:
$$l_q(L(\sigma_E\otimes\beta,t))=l_q(L(\sigma_E\otimes\alpha,t))$$
for every $\beta\in C(\alpha)$. Hence (6.9.1), (6.9.2) and Theorem 5.10 imply that
\begin{eqnarray} l_q(L(\sigma_E\otimes\alpha,t))
&=&\frac{1}{|C(\alpha)|}\sum_{\beta\in C(\alpha)}
l_q(L(\sigma_E\otimes\beta,t))\nonumber\\
&\leq&\frac{1}{|C(\alpha)|}\sum_{\beta\in C(\alpha)}
\dim_{\mathbf k}(H^2(\mathcal E',\mathcal O_{\mathcal E'})
\otimes_{\mathbb F_q}\mathbf k^{\beta})\nonumber\\
&=&\frac{1}{d}\cdot\Delta(\text{Res}_{\mathbb F_q\rightarrow\mathbb F_p}
([H^2(\mathcal E',\mathcal O_{\mathcal E'})])(\alpha)\nonumber\\
&=&\frac{1}{12}\deg(\Delta_E)+v_q(\epsilon(\alpha^{-1})).\nonumber
\end{eqnarray}
The claim now follows from Proposition 4.5.
\end{proof}
\begin{defn} Following [14] we will call a smooth projective variety $V$ defined over a perfect field of characteristic $p$ ordinary in dimension $n$ if the $n$-dimensional Newton and the $n$-dimensional Hodge polygons of $V$ agree. 
\end{defn}
\begin{thm} Assume that the elliptic surface $\mathcal E$ is ordinary in dimension $2$. Then we have:
$$m(E)\geq d(\frac{1}{12}\deg(\Delta_E)+g-1).$$
\end{thm}
\begin{proof} By Deligne's purity theorem we have $\epsilon(1)=\pm q^{g-1}$ where $1$ denotes the trivial representation. Hence in the special case when $G$ is the trivial group, that is, the cover $f$ is the identity map of $\mathcal C$ onto itself, Theorem says 5.10 that
$$\dim_{\mathbb F_q}H^2(\mathcal E,\mathcal O_{\mathcal E})=
\frac{1}{12}\deg(\Delta_E)+g-1.$$
Because $\mathcal E$ is ordinary in dimension $2$ we have
$$l_q(L(E,t))=
l_q(\det(1-F^*t|H^2(\overline{\mathcal E},\mathbb Q_l)))=
\dim_{\mathbb F_q}H^2(\mathcal E,\mathcal O_{\mathcal E})$$
by definition. The claim now follows from Proposition 4.5.
\end{proof}
\begin{rem} It is known that a smooth projective variety $V$ is ordinary if $V$ is a generic curve of genus $g$ (see [4], [16]), a generic abelian variety of dimension $d$ equipped with a polarisation of degree $r$ (see [18], [19]), or a generic complete intersection of
multidegree $(a_1,a_2,\ldots,a_n)$ (see [9]). It is natural
to expect that the same holds for elliptic surfaces with a section. 
More precisely one might conjecture the following. Assume that $p>3$, let $N\geq2$ be a positive integer, let $g\in\mathbb N$ and let $\mathcal M_{g,N,p}$ denote the course moduli representing the functor which associates to every scheme $T$ over Spec$(\mathbb F_p)$ the set of isomorphism classes of smooth families of elliptic surfaces over a smooth curve of genus $g$ over $T$ with discriminant of degree $12N$ in all geometric fibres over $T$, constructed in the paper [24]. Then I expect that for every $p$, $g$ and $N$ as above there is a non-empty open subscheme $\mathcal U$ of $\mathcal M_{g,N,p}$ such that for every geometric point of $\mathcal U$ the corresponding elliptic surface is ordinary. Of course it is enough to show that there a geometric point of $\mathcal M_{g,N,p}$ such that the corresponding elliptic surface is ordinary in dimension two.
\end{rem}

\section{An upper bound in terms of the conductor}

\begin{notn} For every elliptic curve $E$ defined over a field $K$ of characteristic $p$ let $E^{(p)}$ denote pull-back of $E$ with respect to the Frobenius map $K\rightarrow K$ (given by $x\mapsto x^p$). We will call $E^{(p)}$ the Frobenius twist of $E$. The elliptic curve $E^{(p)}$ is in fact defined over the subfield $K^p$ of $p$-th powers. The absolute Frobenius $\mathbf F:E\rightarrow E^{(p)}$ is an isogeny defined over the field $K$. Finally for every cohomology class $c\in H^1(K,\text{\rm Aut}(E))$ let $E_c$ denote twist of $E$ by $c$.
\end{notn}
\begin{defn} Let $K$ be as above and let $E$ be an elliptic curve defined over $K$ such that $j(E)\neq0,1728$. Because $j(E^{(p)})=j(E)^p$ we have $j(E^{(p)})\neq0,1728$. Hence by part $(c)$ of Proposition 1.2 of [26] on pages 325-326 the groups $\text{\rm Aut}(E)$ and $\text{\rm Aut}(E^{(p)})$ are both equal to multiplication by $\pm1$. Therefore there is a unique isomorphism $f:\text{\rm Aut}(E)\rightarrow\text{\rm Aut}(E^{(p)})$ such that $\mathbf F\circ\phi=
f(\phi)\circ\mathbf F$ for every $\phi\in\text{\rm Aut}(E)$. Let $f_*:H^1(K,\text{\rm Aut}(E))\rightarrow H^1(K,\text{\rm Aut}(E^{(p)}))$ denote the isomorphism induced by the identification $f$.
\end{defn}
The following two claims will be useful, and they seem not to be recorded in the literature.
\begin{lemma} Assume that $j(E)\neq0,1728$. Then the elliptic curves $(E_c)^{(p)}$ and $(E^{(p)})_{f_*(c)}$ are isomorphic over $K$ for every $c\in H^1(K,\text{\rm Aut}(E))$.
\end{lemma}
\begin{proof} For every scheme $X$ over Spec$(K)$ let $\overline X$ denote $X\times_{\text{Spec}(K)}\text{Spec}(\overline K)$. For every $\gamma\in\text{Gal}(\overline K|K)$ and $X$ as above let the symbol $\gamma$ also denote the unique endomorphism of $\overline X$ which makes the diagram:
$$\CD\overline X@>\gamma>>\overline X\\
@VVV@VVV\\ \text{Spec}(\overline K)
@>\gamma>>\text{Spec}(\overline K)\endCD$$
commutative. Fix an isomorphism $\phi:\overline E\rightarrow
\overline{E_c}$ over $\overline K$. Then for every $\gamma\in\text{Gal}(\overline K|K)$ the pull-back of the diagram:
$$\CD\overline E@>\phi>>\overline{E_c}@>\gamma>>
\overline{E_c}@>\phi^{-1}>>\overline E@>\gamma^{-1}>>
\overline E\endCD$$
with respect to the Frobenius map $x\mapsto x^p$ of Spec$(\overline K)$ is:
$$\CD\overline{E^{(p)}}@>\psi>>\overline{(E_c)^{(p)}}
@>\gamma>>\overline{(E_c)^{(p)}}@>\psi^{-1}>>
\overline{E^{(p)}}@>\gamma^{-1}>>\overline{E^{(p)}},\endCD$$
where $\psi:\overline{E^{(p)}}\rightarrow
\overline{(E_c)^{(p)}}$ is the unique isomorphism such that $\mathbf F\circ\phi=\psi\circ\mathbf F$. The Aut$(E)$-valued function $\gamma\mapsto\gamma^{-1}\circ\phi^{-1}\circ\gamma\circ\phi$ on Gal$(\overline F|F)$ is a cocycle which represents $c$. Note that
$$\mathbf F\circ\gamma^{-1}\circ\phi^{-1}\circ\gamma\circ\phi=
\gamma^{-1}\circ\psi^{-1}\circ\gamma\circ\psi\circ\mathbf F\quad(\forall\gamma\in
\textrm{Gal}(\overline F|F)),$$
therefore the function $\gamma\mapsto\gamma^{-1}\circ\psi^{-1}\circ\gamma\circ\psi$ is a cocycle which represents $f_*(c)$. The claim is now clear.
\end{proof}
\begin{prop} Let $K$ be a field of characteristic $p$ and let $E$ be an elliptic curve defined over $K$. Assume that $j(E)\neq0,1728$ and $j(E)\in K^p$. Then $E$ is isomorphic to the Frobenius twist of an elliptic curve $E'$ defined over $K$. 
\end{prop}
\begin{proof} Let $\lambda\in K$ be the unique $p$-th root of $j(E)$. By Proposition 1.1 of [26] on pages 324-325 there is an elliptic curve $\widetilde E$ defined over $K$ such that $j(\widetilde E)=\lambda$. Then the Frobenius twist $\widetilde E^{(p)}$ of $\widetilde E$ is defined over $K$ and has the same $j$-invariant as $E$. Hence by the theory of twists there is a cohomology class $c\in H^1(K,\text{\rm Aut}(\widetilde E^{(p)}))$ such that $E$ is the twist of $\widetilde E^{(p)}$ by $c$. Let $E'$ be the twist of $\widetilde E$ by $f^{-1}_*(c)$. By Lemma 7.3 above $E$ is isomorphic to the Frobenius twist of $E'$.
\end{proof}
Now let $K$  denote the function field of a smooth, projective, geometrically irreducible curve $X$ defined over a perfect field of characteristic $p$. 
\begin{cor} Let $E$ be a non-isotrivial elliptic curve defined over $K$. Then $E$ is isogeneous to an elliptic curve $E'$ defined over $K$ with the property $j(E')\notin K^p$.
\end{cor}
\begin{proof} By assumption $j(E)$ does not lie in the constant field of $K$ hence there is a $\lambda\in K$ and natural number $n$ such that $\lambda\notin K^p$ and $\lambda^{p^n}=j(E)$. By Proposition 7.4 there is an elliptic curve $E'$ defined over $K$ such that $j(E')=\lambda$ and $E$ is the $n$-fold Frobenius twist of $E'$. In particular $E$ and $E'$ are isogeneous.
\end{proof}
For every elliptic curve $E$ defined over $K$ let $\Delta_E$ denote the discriminant of a relatively minimal elliptic surface $\mathcal E\rightarrow X$ whose generic fibre is $E$. Then $\Delta_E$ is an effective divisor on the curve $X$. We will say that a non-isotrivial elliptic curve $E$ is minimal in its isogeny class if we have $\deg(\Delta_E)=\min(\deg(\Delta_{E'}))$ where $E'$ is any elliptic curve defined over $K$ isogeneous to $E$. Let $\mathfrak n$ denote the conductor of $E$. 
\begin{thm} {\bf (Pesenti-Szpiro)} Assume that $E$ is an non-isotrivial elliptic curve which is minimal in its isogeny class. With previous notation we have:
$$\deg(\Delta_E)\leq 6(\deg(\mathfrak n)+2g-2).$$
\end{thm}
\begin{proof} By Corollary 7.5 the claim follows at once from
Th\'eor\`eme 0.1 of [22] on page 84 and the isogeny-invariance of the conductor.
\end{proof}
Let us return to the situation of the introduction. Because in each isogeny class of non-isotrivial elliptic curves there is an elliptic curve which is minimal the theorem above combined with Theorem 1.3 has the following immediate
\begin{cor} Assume that $p$ does not divide the order of $\text{\rm Pic}_0(\mathcal C)(\mathbb F_q)$. Then we have:
$$m(E)\leq d(\frac{1}{2}\deg(\mathfrak n)+2g-2)
\text{.}$$
\end{cor}
\begin{rem} Note that this inequality is significantly weaker than Theorem 1.3 because the Pesenti-Szpiro inequality fails to be an equality in general. For example in the special case of the elliptic curve $E$ of Theorem 1.4 a fast inspection of Ulmer's paper [29] (see 2.2-2.3 on page 298) reveals that its conductor $\mathfrak n$ is the sum of the prime divisors of the polynomial $T(1-2^43^3T^n)$. Because by the assumptions of Theorem 1.4 the prime $p$ does not divide $2^43^3n$, the greatest common divisor of $1-2^43^3T^n$ and its derivative is:
$$(1-2^43^3T^n,-n2^43^3T^{n-1})=(1),$$
so we get that the polynomial $1-2^43^3T^n$ is square-fee and therefore $\deg(\mathfrak n)=n+1$. Hence in this case Corollary 7.7 says that
$$m(E)\leq\frac{n}{2}-\frac{3}{2}.$$
On the other hand we will see in the next chapter that in this case Theorem 1.3 says that
$$m(E)\leq\frac{n}{6}-1$$
and the two sides above are actually equal by Theorem 1.4. 
\end{rem}
\begin{rem} Corollary 7.7 was already proved by K.-S. Tan in the special case when $\mathcal C=\mathbb P^1_{\mathbb F_q}$ and $\mathfrak n$ is square-free in [28]. His strategy is similar to ours in reducing the result to estimates of the $p$-adic valuations of coefficients of twisted $L$-functions. For the latter he uses the Grothendieck-Ogg-Shafarevich formula and the functional equation. His methods also use facts about the structure of the set $GL_2(F)\backslash GL_2(\mathbb A)$ specific to the rational function field, hence it is impossible to generalise his approach in order to show Corollary 7.7 in general. 
\end{rem}

\section{Elliptic curves with positive Manin constant}

\begin{defn} Fix a positive integer $n$ which
is not divisible by $p$. Let $E_n$ be the elliptic curve over $F=\mathbb F_q(T)$ (where $q=p^d$) with plane cubic model
$$y^2+xy=x^3-T^n.$$
Straightforward calculation shows that $j(E)^{-1}=T^n(1-2^43^3T^n)$. Hence $E$ is not isotrivial. Let $\Delta_{E_n}$ denote the discriminant of a relatively minimal elliptic surface $\mathcal E_n\rightarrow\mathcal C$ whose generic fibre is $E_n$. The degree of $\Delta_{E_n}$ can be easily computed from the results of 2.2-2.3 of [29] on page 298. In particular when $p\geq5$ it follows at once from these results and Table 4.1 of [27] on page 365 that $\deg(\Delta_{E_n})=12\lceil n/6\rceil$. 
\end{defn}
\begin{defn} Let $F_n$ be the Fermat surface of degree $n$ over $\mathbb F_q$, i.e., the hypersurface in $\mathbb P^3_{\mathbb F_q}$ defined by the equation:
$$x_0^n+x_1^n+x_2^n+x_3^n=0.$$
Let $\mu_n$ denote the group of $n$-th roots of unity in $\overline{\mathbb F}_p$. Let $G$ be the quotient of $\mu_n^4$ modulo the diagonally embedded copy of $\mu_n$. For every $\underline z=(\zeta_0,\zeta_1,\zeta_2,\zeta_3)\in\mu_n^4$ let $[\zeta_0,\zeta_1,\zeta_2,\zeta_3]$ denote the image of $\underline z$ under the quotient map $\mu_n^4\rightarrow G$. Then the group scheme $G$ acts on $F_n$ and the action on the level of points is given by the rule:
$$[\zeta_0,\zeta_1,\zeta_2,\zeta_3]\cdot
[x_0:x_1:x_2:x_3]=
[\zeta_0x_0:\zeta_1x_1:\zeta_2x_2:\zeta_3x_3].$$
Fix a primitive $n$-th root of unity $\zeta\in\overline{\mathbb F}_p$ and let $\Gamma\subset G$ be the subgroup generated by $[\zeta^2,\zeta,1,1]$ and $[1,\zeta,\zeta^3,1]$. As a sub group-scheme $\Gamma$ is defined over $\mathbb F_q$ hence the quotient surface $F_n/\Gamma$ is defined over $\mathbb F_q$, too.
\end{defn}
\begin{thm} The surfaces $\mathcal E_n$ and $F_n/\Gamma$ are birationally equivalent. 
\end{thm}
\begin{proof} This claim is made in 4.3 of [29] on page 301. The proof can be found in sections 3 and 4 of [29] on pages 298--301.
\end{proof}
\begin{notn} Let $\mathbb O_p$ denote the ring of integers of $\overline{\mathbb Q}_p$ and let $\mathfrak p$ be its maximal ideal.  We view all finite fields of characteristic $p$ as subfields of $\mathbb O_p/\mathfrak p$, which is an algebraic closure of $\mathbb F_p$. Reduction modulo $\mathfrak p$ induces an isomorphism between the group of all roots of unity of order prime to $p$ in $\mathbb O_p$ and the multiplicative group of $\mathbb O_p/\mathfrak p$.  We let $\alpha:(\mathbb O_p/\mathfrak p)^*\rightarrow\overline{\mathbb Q}_p^*$ denote the inverse of this isomorphism.  We will use the same letter $\alpha$ for the restriction to any finite field $\mathbb F_q^*$.
\end{notn}
\begin{defn} Fix a nontrivial character $\psi_0:\mathbb F_p\rightarrow\overline{\mathbb Q}_p^*$ and for each finite extension $\mathbb F_{p^m}$ of $\mathbb F_p$, let $\psi:\mathbb F_{p^m}\rightarrow\overline{\mathbb Q}_p^*$ be defined by $\psi=\psi_0\circ\text{\rm Tr}_{\mathbb F_{p^m}|\mathbb F_p}$.  If $\chi:\mathbb F_{p^m}^*\rightarrow\overline{\mathbb Q}^*_p$ is a nontrivial character, we define the corresponding Gauss sum by
$$g(\chi,\psi)=-\sum_{x\in\mathbb F_{p^m}^*}\chi(x)\psi(x).$$
If $\chi_1,\ldots,\chi_r$ are characters $\mathbb F_{p^m}^*\rightarrow\overline{\mathbb Q}^*_p$, not all trivial, such that the product $\chi_1\cdots\chi_r$ is trivial, we define the Jacobi sum $J(\chi_1,\dots,\chi_r)$ by
$$J(\chi_1,\dots,\chi_r)=\begin{array}{ll}
\frac{(-1)^r}{p^m}
\prod_{i=1}^r g(\chi_i,\psi),&\text{if all $\chi_i$ are nontrivial,}\\
0,&\text{otherwise.}\end{array}$$
\end{defn}
\begin{thm} {\bf (Stickelberger)} For any $1\leq k\leq p^m-2$ we have:
$$v_p(g(\alpha^{-k},\psi))=s(k)/(p-1),$$
where if $k=k_0+pk_1+\dots+p^{m-1}k_{m-1}$ is the $p$-adic expansion of the integer $k$, we define $s(k)=k_0+k_1+\dots+k_{m-1}$.
\end{thm}
\begin{proof} This is the second claim of Theorem 9 of [12] on pages 94--95.
\end{proof}
\begin{defn} Let $G$ be the group which we introduced in Definition 8.2.  Let $\widehat  G$ denote the group of characters $G$
with values in $\overline{\mathbb Q}_p$. Using the character
$\alpha:(\mathbb O_p/\mathfrak p)^*\rightarrow\overline{\mathbb Q}^*_p$ we can identify $\widehat  G$ with
$$\left\{a=(a_0,a_1,a_2,a_3)\in (\mathbb Z/n\mathbb Z)^4\left|\sum a_i=
0\right.\right\}$$
where the duality pairing
$G\times\widehat  G\rightarrow\overline{\mathbb Q}^*_p$ is
$$a(z)=
\left\langle(a_0,a_1,a_2,a_3),[\zeta_0,\zeta_1,\zeta_2,\zeta_3]\right\rangle
=\prod_{i=0}^3\alpha(\zeta_i)^{a_i}.$$
For every $a\in\widehat  G$ we let $u(a)$ denote the smallest positive integer such that $q^{u(a)}a=a$. For every non-zero $a=(a_0,a_1,a_2,a_3)\in\widehat  G$ we define the Jacobi sum $J(a)$ as follows: let $\chi_i:\mathbb F_{q^{u(a)}}^*\rightarrow\overline{\mathbb Q}_p$ be defined as $\chi_i=\alpha^{\frac{q^{u(a)}-1}{d}a_i}$ and set $J(a)=J(\chi_0,\dots,\chi_3)$.  Note that $J(qa)=J(a)$.  By convention, we set $J(0)=q$. Let $\Gamma^\bot\subset\widehat  G$ be the cyclic subgroup of order $n$ generated by $(3,-6,2,1)$ and let
$$\widehat G'=\{a=(a_0,\dots,a_3)\in\widehat G\ |\ a=0\text{ or }
a_i\neq0\text{ for }i=0,\dots3\}.$$
\end{defn}
As in Notation 6.6 for every Spec$(\mathbb F_q)$-scheme $S$ let $\overline S$ denote again its base change to Spec$(\overline{\mathbb F}_p)$ and let $F:\overline S\rightarrow\overline S$ denote the Frobenius relative to $\mathbb F_q$ for every such $S$. Choose a prime $l\neq p$ and fix an isomorphism $\nu:\overline{\mathbb Q}_l\rightarrow\overline{\mathbb Q}_p$. We will identify $\overline{\mathbb Q}_l$ with $\overline{\mathbb Q}_p$ via $\nu$.
\begin{thm} Let $A_1,\dots,A_k$ be the orbits of multiplication by $q$ on $\Gamma^\bot\cap\widehat  G'$ and choose $a_i\in A_i$.  Then
$$\det\left(1-F^*\,t|H^2(\overline{F_n/\Gamma},\overline{\mathbb Q}_l)\right)=\prod_{i=1}^k(1-J(a_i)t^{u(a_i)}).$$
\end{thm}
\begin{proof} This is Corollary 7.7 of [29] on page 310.
\end{proof}
\begin{defn} As in Notation 3.1 pick an isomorphism $\iota:\overline{\mathbb Q}_l\rightarrow\mathbb C$ and  identify $\overline{\mathbb Q}_l$ with $\mathbb C$ via $\iota$ in all that follows. Let $|\cdot|_{\infty}$ denote the usual archimedean absolute value on $\mathbb C$. For every $\alpha\in\mathbb R$, finite dimensional $\overline{\mathbb Q}_l$-vector space $V$ and $\overline{\mathbb Q}_l$-linear endomorphism $\Psi:V\rightarrow V$ we say that the pair $(V,\Psi)$ has weights $\leq\alpha$ if for every eigenvalue $\lambda$ of $\Psi$ we have $|\lambda|_{\infty}\leq|q|^{\alpha}$. Moreover we say that $(V,\Psi)$ has slope $\alpha$ if for every $\lambda$ as above $v_q(\lambda)=\alpha$. For every $\alpha\in\mathbb R$, $k\in\mathbb N$ and scheme $X$ of finite type over Spec$(\mathbb F_q)$ we say that $X$ has  weights $\leq\alpha$ in dimension $k$ if the pair $(H^k(\overline X,\overline{\mathbb Q}_l),F^*)$ has weights $\leq\alpha$. Similarly we say that $X$ has slope $\alpha$ in dimension $k$ if the pair $(H^k(\overline X,\overline{\mathbb Q}_l),F^*)$ has slope $\alpha$.
\end{defn}
Now let $C$ be a curve over $\mathbb F_q$, i.e. a one-dimensional (but necessarily equidimensional) variety defined
over $\mathbb F_q$.
\begin{lemma} The following holds:
\begin{enumerate}
\item[$(i)$] the curve $C$ has weights $\leq1$ in dimension $1$,
\item[$(ii)$] the curve $C$ has slope $1$ in dimension $2$.
\end{enumerate}
\end{lemma}
\begin{proof} Let $U$ be a smooth dense open subscheme of $C$ and let $i:U\rightarrow C$ be the inclusion map. Let $S$ be the reduced closed subscheme
of $C$ whose underlying set is the complement of $U$ and let $j:S\rightarrow C$ be the inclusion map. Then there is a cohomological long exact sequence:
$$\CD H^1(\overline C,i_!(\overline{\mathbb Q}_l))
@>\alpha>>H^1(\overline C,\overline{\mathbb Q}_l)
@>>> H^1(\overline C,j_*(\overline{\mathbb Q}_l))@>>> \cdots\endCD$$
$$\CD\cdots@>>>H^2(\overline C,i_!(\overline{\mathbb Q}_l))@>\beta>>
H^2(\overline C,\overline{\mathbb Q}_l)@>>> 
H^2(\overline C,j_*(\overline{\mathbb Q}_l)).\endCD$$
Because $j_*(\mathbb Q_l)$ is the direct sum of skyskaper sheaves, it is acyclic. Hence the map $\alpha$ is surjective and the map $\beta$ is an isomorphism. By the proper base change theorem $H^1_c(\overline U,\overline{\mathbb Q}_l)=
H^1(\overline C,i_!(\overline{\mathbb Q}_l))$ so the pair $(H^1(\overline C,
i_!(\overline{\mathbb Q}_l)),F^*)$ has weights $\leq1$ by Deligne's purity theorem. Similarly $H^2_c(\overline U,\overline{\mathbb Q}_l)=H^2(\overline C,i_!(\overline{\mathbb Q}_l))$ so the pair $(H^2(\overline C,i_!(\overline{\mathbb Q}_l),F^*)$ has slope $1$ by the duality theorem. Because the maps $\alpha$ and $\beta$ are $F^*$-equivariant the claims are now clear.
\end{proof}
\begin{notn} Now let $X$ be a surface over $\mathbb F_q$, i.e. a two-dimensional variety defined over $\mathbb F_q$. Let $U$ be a smooth dense open sub-scheme of $X$ and let $i:U\rightarrow X$ be the inclusion map. Then we have a map:
$$i_*:H^2_c(\overline U,\overline{\mathbb Q}_l)\longrightarrow
H^2(\overline X,\overline{\mathbb Q}_l)$$
which is the composition of the isomorphism $H^2_c(\overline U,\overline{\mathbb Q}_l)=H^2(\overline X,i_!(\overline{\mathbb Q}_l))$ furnished by proper base change and the homomorphism $H^2(\overline X,i_!(\overline{\mathbb Q}_l))\rightarrow H^2(\overline X,\overline{\mathbb Q}_l)$ induced by the inclusion
$i_!(\overline{\mathbb Q}_l)\subseteq\overline{\mathbb Q}_l$.
\end{notn}
\begin{lemma} The following holds:
\begin{enumerate}
\item[$(i)$] the pair $(\text{\rm Ker}(i_*),F^*)$ has weights $\leq1$,
\item[$(ii)$] the pair $(\text{\rm Coker}(i_*),F^*)$ has slope $1$.
\end{enumerate}
\end{lemma}
\begin{proof} Let $C$ be the reduced closed subscheme whose underlying set is the complement of $U$ and let $j:C\rightarrow X$ be the inclusion map. Then $C$ is a curve and there is a cohomological long exact sequence:
$$\CD H^1(\overline X,j_*(\overline{\mathbb Q}_l))
\rightarrow H^2(\overline X,i_!(\overline{\mathbb Q}_l))\rightarrow 
H^2(\overline X,\overline{\mathbb Q}_l)
\rightarrow H^2(\overline X,j_*(\overline{\mathbb Q}_l)).\endCD$$
Because $H^1(\overline X,j_*(\overline{\mathbb Q}_l))=
H^1(\overline C,\overline{\mathbb Q}_l)$ and  $H^2(\overline X,j_*(\overline{\mathbb Q}_l))=H^2(\overline C,\overline{\mathbb Q}_l)$, the pair
$(H^1(\overline X,j_*(\overline{\mathbb Q}_l)),F^*)$ has weights $\leq1$ and the pair $(H^2(\overline X,j_*(\overline{\mathbb Q}_l)),F^*)$
has slope $1$ by Lemma 8.10. The claims are now clear.
\end{proof}
\begin{prop} Let $X_1$ and $X_2$ be two birationally equivalent geometrically irreducible projective surfaces over $\text{\rm Spec}(\mathbb F_q)$. Assume that both $X_1$ and $X_2$ are pure of weight $2$ in dimension $2$. Then
$$l_q(\det\left(1-F^*\,t|H^2(\overline X_1,\overline{\mathbb Q}_l)\right))=
l_q(\det\left(1-F^*\,t|H^2(\overline X_2,\overline{\mathbb Q}_l)\right)).$$
\end{prop}
\begin{proof} Let $U$ be a smooth, two-dimensional scheme over Spec$(\mathbb
F_q)$ such that there are open immersions $i_1:U\rightarrow X_1$ and $i_2:U\rightarrow X_1$. Then
$$l_q(\det\left(1-F^*\,t|\text{Im}(i_{k*}\right))=
l_q(\det\left(1-F^*\,t|H^2(\overline X_k,\overline{\mathbb Q}_l)\right))$$
for $k=1,2$, using the notation of 8.11, by part $(ii)$ of Lemma 8.12. Let $V$ denote the largest $F^*$-invariant $\overline{\mathbb Q}_l$-linear subspace of $H^2_c(\overline U,\overline{\mathbb Q}_l)$ which has weights $\leq1$. Because $H^2(\overline X_k,\overline{\mathbb Q}_l)$ is pure of weight $2$ we have $V\subseteq\text{Ker}(i_{k*})$ for $k=1,2$. But $V\supseteq\text{Ker}(i_{k*})$ by part $(i)$ of Lemma 8.12. Hence
$$\det\left(1-F^*\,t|\text{Im}(i_{1*})\right)=
\det\left(1-F^*\,t|H^2_c(\overline U,\overline{\mathbb Q}_l)/V\right)=
\det\left(1-F^*\,t|\text{Im}(i_{2*})\right),$$
so the claim is clear.
\end{proof}
Assume now that $q=p$. 
\begin{thm} Let $p$ be a prime number and let $n$ be a positive integer as above. Assume that $n|p-1$ and $6|n$. Then $E_n$ is not isotrivial and 
$$m(E_n)=\frac{n}{6}-1=\frac{1}{12}\deg(\Delta_{E_n})-1.$$
\end{thm}
\begin{proof} Let $\underline a=(-3,6,-2,-1)\in\widehat G$. By our assumption for every $k=0,1,\ldots,n-1$ we have $k\underline a\in\widehat G'$ if and only if $k\neq n/6,n/3,n/2,2n/3$ or $5n/6$. Because of our assumption $n|p-1$ every orbit of multiplication by $p$ on $\Gamma^\bot\cap\widehat  G'$ consists of one element. Hence
\begin{equation}
\det\left(1-F^*\,t|H^2(\overline{F_n/\Gamma},\overline{\mathbb Q}_l)\right)=\!\!\!\!\!
\prod_{\substack{0\leq k\leq n-1\\ k\neq n/6,n/3,n/2,2n/3,5n/6}}
\!\!\!\!\!(1-J(k\underline a)t)
\end{equation}
by Theorem 8.8. Gauss sums are Weil numbers of weight $1$ hence the reciprocal roots of the polynomial in (8.14.1) above are Weil numbers of weight $2$. The surface $\mathcal E_n$ is smooth, so it is pure of weight $2$ in dimension $2$ by Deligne's purity theorem. Hence by Theorem 8.3 and Proposition 8.13 we have:
$$l_q(L(E_n,t))=
l_q(\det\left(1-F^*\,t|H^2(\overline{\mathcal E_n},\overline{\mathbb Q}_l)
\right))=
l_q(\det\left(1-F^*\,t|H^2(\overline{F_n/\Gamma},\overline{\mathbb Q}_l)\right)).$$
For every $k=1,2,\ldots,n/6-1$ we have:
\begin{eqnarray}
v_p(J(k\underline a))\!\!\!\!\!\!&=&\!\!\!\!
v_p\!\left(g(\alpha^{\frac{-3k(p-1)}{n}},\psi)
g(\alpha^{\frac{6k(p-1)}{n}-p+1},\psi)
g(\alpha^{\frac{-2k(p-1)}{n}},\psi)
g(\alpha^{\frac{-k(p-1)}{n}},\psi)\right)-1\nonumber\\
\!\!\!\!\!\!&=&\!\!\!\!\frac{3k}{n}+
1-\frac{6k}{n}+
\frac{2k}{n}+
\frac{k}{n}-1=0\nonumber
\end{eqnarray}
by Stickelberger's Theorem 8.6, since every exponent of $\alpha^{-1}$ in the equation above is a positive integer strictly less than $p$. Hence by Proposition 4.5 and the above:
$$m(E_n)\geq l_q(L(E_n,t))\geq\frac{n}{6}-1.$$
Because $6|p-1$ by our assumptions, we have $p\geq7$. Hence
$$m(E_n)\leq\frac{1}{12}\deg(\Delta_{E_n})-1=\frac{n}{6}-1$$
by Theorem 6.9 and by our remark at the end of Definition 8.1. The claim of the theorem above is now clear.
\end{proof}

\section{Strong Weil curves}

\begin{defn} Fix now a closed point $\infty$ of $\mathcal C$ and let $A$ denote the ring of rational functions on $\mathcal C$ regular away from $\infty$ as in the introduction. For any non-zero ideal $\mathfrak m$ of $A$ an irreducible affine algebraic curve $Y_0(\mathfrak m)$ is defined over $F$, the Drinfeld modular curve parameterising Drinfeld $A$-modules of rank two of generic characteristic with Hecke level $\mathfrak m$-structure. There is a unique non-singular projective curve $X_0(\mathfrak m)$ over $F$ which contains $Y_0(\mathfrak m)$ as a dense open subvariety. Let $J_0(\mathfrak m)$ denote the Jacobian of the curve $X_0(\mathfrak m)$. The ideals of $A$ and the effective divisors on $\mathcal C$ whose support does not contain $\infty$ are in a natural one to one correspondence, and we will not distinguish them in all that follows. Let $E$ be an elliptic curve defined over $F$ which has split multiplicative reduction at $\infty$. Then its conductor is of the form $\mathfrak m\infty$, where $\mathfrak m$ is an ideal of $A$. By the function field analogue of the Taniyama-Weil conjecture there is a non-constant map $\pi:X_0(\mathfrak m)\rightarrow E$ defined over $F$. For any map $h:X_0(\mathfrak m)\rightarrow C$, where $C$ is an elliptic curve, let $h_*:J_0(\mathfrak m)\rightarrow C$ and $h^*:C\rightarrow J_0(\mathfrak m)$ denote the maps induced by the Albanese and the Picard functoriality, respectively.
\end{defn}
\begin{thm} The following are equivalent:
\begin{enumerate}
\item[$(i)$] the kernel of the map $\pi_*:J_0(\mathfrak m)\rightarrow E$ is an abelian variety,
\item[$(ii)$] the map $\pi^*:E\rightarrow J_0(\mathfrak m)$ is a closed
immersion,
\item[$(iii)$] the degree of $\pi$ is minimal among all non-degenerate
maps $\rho:X_0(\mathfrak m)\rightarrow E'$, where $E'$ is any elliptic curve isogeneous to $E$ over $F$.
\end{enumerate}
\end{thm}
\begin{proof} This result is well known in the mathematical folklore,
but it is difficult to track down a proof (the standard reference, Lemme
3 of [15] on pages 282-283 only shows the equivalence of $(i)$ and
$(iii)$). Our excuse for giving a full proof other than the above is that
we consider the more delicate case of positive characteristic. First
assume that $(i)$ holds. By the multiplicity one theorem for the action of the Hecke algebra on $J_0(\mathfrak m)$ the abelian variety Ker$(\pi_*)$ has no quotient isogeneous to $E$, so for every non-degenerate map $\rho:X_0(\mathfrak m)\rightarrow E'$, where $E'$ is any elliptic curve isogeneous to $E$ over $F$, the kernel of $\rho_*$ must contain Ker$(\pi_*)$. Hence the map $\rho$ factors through $\pi$, in particular its degree is at least as big as the degree of $\pi$. On the other hand if $(iii)$ holds then Ker$(\pi_*)$ contains the reduction of the connected component of its identity element as a closed subgroup-scheme. The quotient  $E'$ of $J_0(\mathfrak m)$ by the latter is an elliptic curve isogeneous to $E$, hence the degree the corresponding map $\rho:X_0(\mathfrak m)\rightarrow E'$ is at least as big as $\deg(\pi)$. But $\pi$ factors through $\rho$, so they must be equal. Note that the map $\pi^*$ is just the dual of the morphism $\pi_*:J_0(\mathfrak m)\rightarrow E$ of the principally polarised abelian varieties. The equivalence of $(i)$ and $(ii)$ and therefore the theorem itself now follows from the lemma below.
\end{proof}
\begin{lemma} Let $\phi:A\rightarrow B$ be a surjective
homomorphism of abelian varieties and let $\phi^{\vee}:B^{\vee}\rightarrow
A^{\vee}$ be its dual. Then the following are equivalent:
\begin{enumerate}
\item[$(i)$] the kernel of $\phi$ is an abelian variety,
\item[$(ii)$] the map $\phi^{\vee}$ is a closed
immersion.
\end{enumerate}
\end{lemma}
\begin{proof} This proof was explained to me by Laurent Fargues.
For any $S$-scheme $T$ let $T$ also denote the sheaf represented by $T$ on the {\it fppf} topology on $S$. Attached to the short exact
sequence
$$\CD0@>>>\text{Ker}(\phi)@>>> A@>\phi>> B@>>>0\endCD$$
of sheaves on the {\it fppf} topology there is a cohomological exact
sequence:
$$\CD\text{Hom}(A,\mathbb G_m)\rightarrow 
\text{Hom}(\text{Ker}(\phi),\mathbb G_m)@>>> 
\text{Ext}^1(B,\mathbb G_m)@>\phi^{\vee}>>\text{Ext}^1(A,\mathbb G_m).\endCD$$
By a theorem of Grothendieck for any abelian scheme $C$ the sheaf
$\text{Ext}^1(C,\mathbb G_m)$ is represented by the dual of $C$ and for any
morphism $\phi:A\rightarrow B$ of abelian varieties the induced map 
$\text{Ext}^1(B,\mathbb G_m)\rightarrow \text{Ext}^1(A,\mathbb G_m)$ is the dual of $\phi$ as the notation above indicates. Moreover the sheaf Hom$(C,\mathbb G_m)$ is trivial for any abelian scheme $C$ hence $\phi^{\vee}$ is an immersion if and only if Hom$(\text{Ker}(\phi),\mathbb G_m)$ is trivial. Let $\text{Ker}(\phi)_0$ denote the reduced group scheme associated to the connected component of Ker$(\phi)$, considered as a closed subgroup-scheme. It is an abelian subscheme of $\text{Ker}(\phi)$ such that the quotient
$G=\text{Ker}(\phi)/\text{Ker}(\phi)_0$ is a finite, flat group scheme. By
looking at the cohomological exact sequence attached to the short exact sequence
$$\CD0@>>>\text{Ker}(\phi)_0@>>>
\text{Ker}(\phi)@>>>G@>>>0\endCD$$
of sheaves on the {\it fppf} topology, we get that Hom$(\text{Ker}(\phi),\mathbb G_m)=\text{Hom}(G,\mathbb G_m)$. The sheaf
Hom$(G,\mathbb G_m)$ is represented by the Cartier dual of the group scheme $G$, so it is trivial if and only if $G$ is trivial.
\end{proof}
\begin{defn} If the equivalent conditions of the theorem
above hold then we say that $E$ is a strong Weil curve and the modular parameterisation $\pi:X_0(\mathfrak m)\rightarrow E$ is optimal. By the proof above it is clear that up to isomorphism $E$ is unique in its isogeny class and there is only one strong Weil map parameterising $E$. On the other hand by property $(i)$ the quotient of $J_0(\mathfrak m)$ by the reduced group scheme associated to the connected component of the kernel of the map $\pi_*:J_0(\mathfrak m)\rightarrow E$ induced by any modular parameterisation $\pi:X_0(\mathfrak m)\rightarrow E$ is a strong Weil curve. Hence there is a strong Weil curve in the isogeny class of every elliptic curve having split multiplicative reduction at $\infty$.
\end{defn}

\section{Applications to the degree conjecture}

\begin{defn} For any graph $G$ let $\mathcal V(G)$ and $\mathcal E(G)$ denote its set of vertices and edges, respectively. Let $R$ be a commutative group and let $G$ be a locally finite oriented graph. In this paper we will assume that every oriented graph $G$ is equipped with an involution $\overline{\cdot}:\mathcal E(G)\rightarrow\mathcal E(G)$ such that for each edge $e\in\mathcal E(G)$ the original and terminal vertices of the edge $\overline e\in\mathcal E (G)$ are the terminal and original vertices of $e$, respectively. The edge $\overline e$ is called the edge $e$ with reversed orientation. If for each edge $e\in\mathcal E(G)$ there is exactly one edge $\overline e\in\mathcal E (G)$ whose original and terminal vertices are the terminal and original vertices of $e$ then there is a unique involution of this type. The Bruhat-Tits tree $\mathcal T$ is such a graph. A function $\phi:\mathcal E(G)\rightarrow R$ is called a harmonic $R$-valued cochain, if it satisfies the following conditions:
\begin{enumerate}
\item[$(i)$] We have:
$$\phi(e)+\phi(\overline e)=0\text{\ }(\forall e\in\mathcal E(G)).$$
\item[$(ii)$] If for an edge $e$ we introduce the notation $o(e)$ and
$t(e)$ for its original and terminal vertex respectively,
$$\sum_{\substack{e\in\mathcal E(G)\\o(e)=v}}\phi(e)=0\text{\ }
(\forall v\in\mathcal V(G)).$$
\end{enumerate}
We denote by $H(G,R)$ the group of $R$-valued harmonic cochains on $G$. 
\end{defn}
\begin{defn} Let $Y\subset F^2$ be an $A$-lattice, that is, a projective $A$-submodule of rank two. Let $\Gamma(Y)$ denote the $F$-linear automorphisms of $F^2$ leaving $Y$ invariant and for every ideal $\mathfrak a \triangleleft A$ let $\Gamma(Y,\mathfrak a)\leq\Gamma(Y)$ denote the subgroup of those elements who induce the identity on the quotient $A$-module $Y/\mathfrak aY$. We say that a subgroup $\Gamma$ of $GL_2(F)$ is arithmetic if there is an $A$-lattice $Y$ and an ideal $\mathfrak a$ such that $\Gamma$ is contained in $\Gamma(Y)$ and it contains $\Gamma(Y,\mathfrak a)$. Let $\Gamma$ be an arithmetic subgroup of $GL_2(F)$ and let $F_{\infty}$ denote the completion of $F$ with respect to the valuation corresponding to $\infty$. As a subgroup of $GL_2(F_{\infty})$ the arithmetic group $\Gamma$ acts on the Bruhat-Tits tree associated to $PGL_2(F_{\infty})$ on the left, which we will denote by $\mathcal T$. For any abelian group $M$ let $H_!(\mathcal T,M)^{\Gamma}$ denote the group of those $\Gamma$-invariant, $M$-valued harmonic cochains on $\mathcal T$ which has finite support as a function on the edges of the quotient graph $\Gamma\backslash\mathcal T$.
\end{defn}
\begin{defn} Next we are going to define the first homology group $H_1(\Gamma\backslash\mathcal T,\mathbb Z)$. The latter has several descriptions. It may be defined as the first topological homology group of the CW-complex attached to $\Gamma\backslash\mathcal T$ with integral coefficients. It is also canonically isomorphic to the abelianization of the quotient $\Gamma^*=\Gamma/\Gamma_f$, where $\Gamma_f$ be the normal subgroup of $\Gamma$ generated by the elements of finite order. We will use a third, purely combinatorial description, since it is the most convenient for our purposes. Recall that a path on an oriented graph $G$ is a sequence of edges $e_1,e_2,\ldots,e_n\in\mathcal E(G)$ such that $t(e_i)=o(e_{i+1})$ for $i=1,2,\ldots,n-1$. The path is closed if the equality $t(e_n)=o(e_1)$ holds, too. For each edge $e\in\mathcal E(G)$ let $i_e:\mathcal E(G)\rightarrow\mathbb Z$ denote the unique function such that
$$i_e(f)=
\left\{\begin{array}{ll}+1,&\text{if $f=e$,}\\
-1,&\text{if $f=\overline e$,}\\
\quad\!\!0,&\text{otherwise.}\end{array}\right.$$
To any closed path $e_1,e_2,\ldots,e_n$ we associate the function
$\sum_{i=1}^ni_{e_i}$. We define $H_1(G,\mathbb Z)$ as the abelian group of
$\mathbb Z$-valued functions on $\mathcal E(G)$ generated by these functions.
Let us return to the special case $G=\Gamma\backslash\mathcal T$. Let
$z(\Gamma)$ denote the cardinality of the center of $\Gamma$ and
let $\Gamma_e$ be the stabiliser of the edge $e\in\mathcal E(\mathcal T)$ in $\Gamma$. Let us recall rapidly why $\Gamma_e$ is finite. Let
$v:GL_2(F_{\infty})\rightarrow\mathbb Z$ the composition of the determinant and the valuation, and let $GL_2(F_{\infty})_0$ denote its kernel. We claim that every arithmetic group $\Gamma$ lies in $GL_2(F_{\infty})_0$. Clearly it is enough to show this for $\Gamma(Y)$. The localisation of $Y$ at each prime of $A$ is a free module of rank two, so the determinant of every element of $\Gamma(Y)$ is a unit at each prime of $A$, so it is in fact a unit of $A$. The latter are constants, so they have valuation zero. On the other hand the stabilizator of $e$ in $GL_2(F_{\infty})_0$ is compact, so $\Gamma_e$ is finite as the intersection of a compact and a discrete group. We define 
$$j_{\Gamma}:H_1(\Gamma\backslash\mathcal T,\mathbb Z)\rightarrow
H_!(\mathcal T,\mathbb Z)^{\Gamma},$$
as the map $\phi\mapsto\phi^*$ given by the rule
$\phi^*(e)=|\Gamma_e|\phi(\widetilde e)/z(\Gamma)$, where
$\widetilde e$ is the image of the edge $e$ in $\mathcal
E(\Gamma\backslash\mathcal T)$. It is easy to see that the homomorphism is
well-defined, that is $\phi^*$ is indeed a harmonic cochain. 
\end{defn}
\begin{prop} The homomorphism is $j_{\Gamma}$ is injective with finite cokernel of exponent dividing $q^{\deg(\infty)}-1$. 
\end{prop}
\begin{proof} The injectivity is trivial if the definition above is
employed. Gekeler and Reversat proved that the cokernel has index prime to the characteristic $p$ (see Proposition 6.4.4 of [5], pages 73-74). A careful reading of the proof of Lemma 3.3.3 of [5], pages 50-51, reveals that our proposition above has already been proved there, if one also uses the result quoted above. Here we reproduce their argument for the sake of the reader. Let $T$ be a maximal tree in the connected graph $\Gamma\backslash\mathcal T$: the former exists by the Zorn lemma. By Serre's structure theorem the graph $\Gamma\backslash\mathcal T$ is the union of a finite graph and finitely many ends, hence the complement $\mathcal E(\Gamma\backslash\mathcal T)-\mathcal E(T)$ is finite. Let $R=\{\widetilde e_1,\ldots,\widetilde e_g\}$ be a set of representatives of $\mathcal E(\Gamma\backslash\mathcal T)-\mathcal E(T)$ modulo orientation, that is for every edge $e\in\mathcal E(\Gamma\backslash\mathcal T)-\mathcal E(T)$ exactly one of the edges $e$ and $\overline e$ is listed in $R$. For each edge $\widetilde e_i\in R$ let $c_i$ denote a closed path consisting of $\widetilde e_i$ and a path connecting $t(e_i)$ with $o(e_i)$ in $T$. For any $\phi\in H_!(\mathcal T,\mathbb Z)^{\Gamma}$ the function
$$\phi-\sum_{i=1}^g\frac{\phi(e_i)z(\Gamma)}{|\Gamma_{e_i}|}j_{\Gamma}(c_i)$$
vanishes identically outside of the maximal tree $T$, hence vanishes
everywhere. Therefore the cokernel of $j_{\Gamma}$ is annihilated by the smallest common multiple of the natural numbers $|\Gamma_{e_i}|
/z(\Gamma)$. Since the torsion of the stabiliser of any edge $e\in\mathcal
E(\mathcal T)$ in the image of $GL_2(F_{\infty})_0$ in $PGL_2(F_{\infty})$
modulo its $p$-torsion group is a  group of order $q^{\deg(\infty)}-1$, the
claim is now clear.
\end{proof}
\begin{notn} In the rest of the paper we assume that $E$ is a strong Weil curve and $\pi:X_0(\mathfrak m)\rightarrow E$ is optimal. Let $\Gamma$ be the arithmetic group:
$$\left\{\left(\begin{matrix} a&b\\c&d\end{matrix}\right)\in GL_2(A)|
c\in\mathfrak m\right\}.$$
One may associate to the strong Weil curve $E$ an element $v_E\in H_!(\mathcal T,\mathbb Z)^{\Gamma}$ lying in the image of the map $j_{\Gamma}$ (for its definition see 3.4 of [21] on page 3493). As it is well-known the set $\mathcal E(\mathcal T)$ can be identified with $GL_2(F_{\infty})/\Gamma_{\infty}Z(F_{\infty})$ where $\Gamma_{\infty}$ is the Iwahori subgroup:
$$\Gamma_{\infty}=\left\{\left(\begin{matrix} a&b\\
c&d\end{matrix}\right)\in GL_2(\mathcal O_{\infty})|
\infty(c)>0\right\}.$$
There is a natural map:
$$h:\mathcal E(\mathcal T)\rightarrow GL_2(F)\backslash GL_2(\mathbb A)/\mathbb K_0(\mathfrak m\infty)Z(\mathbb A)$$
which for $g\in GL_2(F_{\infty})$ maps the left $\Gamma_{\infty}Z(F_{\infty})$-coset of $g$ to the the double coset $GL_2(F)\mathbf gK_0(\mathfrak m\infty)Z(\mathbb A)$ of the unique element $\mathbf g\in GL_2(\mathbb A)$ such that for every $x\in|\mathcal C|$ the $x$-th component of $\mathbf g$ is $g$, if $x$ is $\infty$, and it is $1$, otherwise. By 9.1 of [5] on page 84 the function $v_E$ lies in the $\mathbb Q$-module
$\mathbb Q(\psi_E\circ h)$ spanned by $\psi_E\circ h$. Let $\widetilde c(E)$ be the unique non-negative number such that $v_E=\widetilde c(E)\psi_E\circ h$. By definition $v_E$ generates the $ \mathbb Z$-module $\textrm{Im}(j_{\Gamma})\cap\mathbb Q(\psi_E\circ h)$. Hence Proposition 10.4 has the following immediate
\end{notn}
\begin{cor} We have:
$$\widetilde c(E)\leq(q^{\deg(\infty)}-1)c(E)^{-1}\text{.}$$
\end{cor}
Using the Riemann hypothesis for the $L$-function $L(\text{Symm}^2(E),t)$ of the second symmetric square of the Galois representation $H^1(E_{ \overline F},\mathbb Q_l)$ (see Theorem 4.6 of [21] on pages 3501-3502) Papikian deduces from his main formula for the degree of modular parameterisations of elliptic curves, Proposition 1.3 of the same paper on page 3485, the following:
\begin{thm} Assume that $\mathfrak m$ is square-free. Then we have:
$$\deg(\pi)<\widetilde c(E)^2\cdot q^{14g+\deg(\infty)+5}\cdot q^{\deg(\mathfrak m)}\cdot
\deg(\mathfrak m)^3\text{.}$$ 
\end{thm}
Combining Corollary 7.7 with Corollary 10.6 and Theorem 10.7 we get the following result:
\begin{thm} Assume that $p$ does not divide the order of $\text{\rm Pic}_0(\mathcal C)(\mathbb F_q)$. Also suppose that $\mathfrak m$ is square-free. Then we have:
$$\deg(\pi)<q^{18g+4\deg(\infty)+1}\cdot q^{2\deg(\mathfrak m)}\cdot
\deg(\mathfrak m)^3\text{.}$$
\end{thm}

\end{document}